\documentclass[11pt]{article}
\usepackage[T2A]{fontenc}
\usepackage{physics}
\usepackage{amsmath}
\usepackage{tikz}
\usepackage{mathdots}
\usepackage{yhmath}
\usepackage{cancel}
\usepackage{color}
\usepackage{gensymb}
\usepackage{siunitx}
\usepackage{array}
\usepackage{multirow}
\usepackage{amssymb}

\usepackage{tabularx}
\usepackage{extarrows}
\usepackage{booktabs}
\usetikzlibrary{fadings}
\usetikzlibrary{patterns}
\usetikzlibrary{shadows.blur}
\usetikzlibrary{shapes}
\usepackage[english]{babel}
\usepackage{mathbbol}
\usepackage{amsthm,amssymb}

\usepackage{hyperref}
\usepackage{enumerate}

\title{The infinitesimal behavior of the sum of Cauchy kernels and its derivative at infinity}
\author{Vladimir Shemyakov \\
St. Petersburg, Russia \\
	\texttt{vladimir.v.shemyakov@gmail.com}
}
\date{\today}

\begin{document}

\maketitle

\newtheorem{definition}{Definition}
\newtheorem{Remark}{Remark}
\newtheorem{lema}{Lemma}
\newtheorem{thm}{Theorem}

\begin{abstract}
In analysis, it's often useful to know the value of a function at infinity, this operation possesses pleasant properties. However, even when the limit does not exist, some intuitive considerations may suggest that the function still assumes a specific value at infinity in a certain sense. In Nevanlinna theory, all objects are studied on average, i.e., their integrals, hence the integral interpretation of the concept of convergence to a limit is beneficial for the theory of meromorphic functions. This is precisely the focus of this work, applied to sums of Cauchy kernels and their derivatives.
\end{abstract}

\section{Introduction}
\begin{definition}
Let $\displaystyle \{c_{n}\}_{n\in \mathbb{N}} ,\{t_{n}\}_{n\in \mathbb{N}} \subset \mathbb{C}$, such that $\displaystyle \lim _{n\rightarrow \infty } |t_{n} |=\infty $ and $\displaystyle \{t_{n}\}_{n\in \mathbb{N}}$ has no finite limit points and 
\begin{equation*}
\sum _{n=1}^{\infty }\frac{|c_{n} |}{|t_{n} |} < +\infty \ -\ \text{the natural condition.}
\end{equation*}
Define a meromorphic function in $\displaystyle \mathbb{C}$ as $\displaystyle \sum _{n=1}^{\infty }\frac{c_{n}}{z-t_{n}}$ - the sum of Cauchy kernels.
\end{definition}
\begin{Remark}
The series converges uniformly on compact sets in $\displaystyle \mathbb{C}$ that do not contain poles, and it can be differentiated term by term infinitely many times. 
\end{Remark}
Here are some well-known examples of Cauchy kernel sums:
\begin{equation*}
\sum _{n=-\infty }^{\infty }\frac{( -1)^{n}}{z-n} =\frac{\pi }{\sin( \pi z)} ,\ \ \sum _{n=-\infty }^{\infty }\frac{( -1)^{n}}{n(z-n)} =\frac{\pi }{z\sin\left( \pi z\right)}-\frac{1}{z^2}, 
\end{equation*}
\begin{equation*}
\sum _{n=1}^{\infty }\frac{1}{n( z-n)} =\frac{\gamma +\Psi ( 1-z)}{z} ,\ \ \sum _{n=1}^{\infty }\frac{1}{n^{2}( z-n)} =\frac{\pi ^{2}}{6} +\frac{\gamma +\Psi ( 1-z)}{z^{2}},
\end{equation*}
\begin{equation*}
\sum _{n=1}^{\infty }\frac{1}{z-n^{2}} =\frac{\pi \cot\left( \pi \sqrt{z}\right)\sqrt{z} -1}{2z} ,\ \ \ \sum _{n=1}^{\infty }\frac{1}{2^{n}( z-n)} =-\frac{\Phi \left(\frac{1}{2} ,1,-z\right) z+1}{z},
\end{equation*}
where $\displaystyle \Psi =\frac{\Gamma '}{\Gamma }$, $\displaystyle \Gamma $ - Euler's Gamma function, $\displaystyle \Phi $ - Lerch's zeta function.

For such functions, it is clear that a limit cannot exist at infinity:
\begin{equation*}
\nexists \lim _{z\rightarrow \infty } \ \sum _{n=1}^{\infty }\frac{c_{n}}{z-t_{n}},
\end{equation*}
since there are infinitely many poles in any neighborhood of $\displaystyle \infty $, and the function takes on infinite values at them. But for finite sums, the limit is zero:
\begin{equation*}
\lim _{z\rightarrow \infty } \ \sum _{n=1}^{N}\frac{c_{n}}{z-t_{n}} =0.
\end{equation*}
Therefore, intuition suggests that in the proper sense, the function should still tend to zero.

M.V. Keldysh in the work (\cite{Keldysh}, 6, p. 159) showed that
\begin{equation*}\label{eq:keld}
\int\limits _{0}^{2\pi }\ln^{+}\left| \sum _{n=1}^{\infty }\frac{c_{n}}{re^{i\varphi } -t_{n}}\right| d\varphi \leqslant 2\left(1+\frac{1}{r}\sum\limits_{|t_n|<r}|c_n|+\sum\limits_{|t_n|>r}\frac{|c_n|}{|t_n|}\right)=O(1),
\end{equation*}
where $\ln^{+} =\max(1,\ln)$ and integration is carried out where the modulus of the function $\displaystyle >1$. This integral is important in the theory of value distribution of meromorphic functions and is one of the components of the Nevanlinna characteristic. Thanks to Chebyshev's inequality, a pointwise statement can be obtained
\begin{gather*}
\forall \{r_{k}\}_{k\in \mathbb{N}}  >0\ \text{strictly increasing},\ \forall \varepsilon \in ( 0;2\pi ) \ \  \exists \{E_{k}\}_{k\in \mathbb{N}} -\ \text{measurable},\ E_{k} \subset [ 0,2\pi ] \ \text{и} \ |E_{k} | >2\pi -\varepsilon  \\
\text{such that} \ \ \forall k\in \mathbb{N} \ \ \ \forall \varphi \in E_{k} \ \ \ \left| f\left( r_{k} e^{i\varphi }\right)\right| \leqslant e^{\frac{M}{\varepsilon }},
\end{gather*}
where $\displaystyle M$ denotes the supremum of this integral, and $\displaystyle |E|$ denotes the Lebesgue measure of the set $\displaystyle E$. The figure illustrates the idea that the function is bounded on each circle outside of a small 'bad' set.

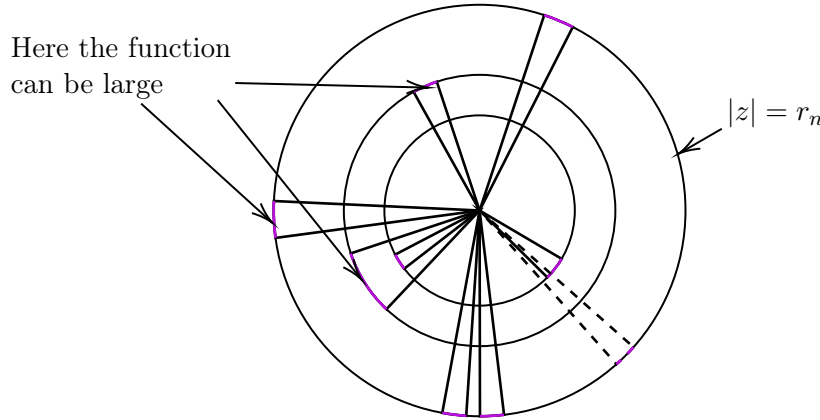
\tikzset{every picture/.style={line width=0.75pt}} 

\begin{tikzpicture}[x=0.75pt,y=0.75pt,yscale=-1,xscale=1]
\path (40,230);
\draw  [color={rgb, 255:red, 0; green, 0; blue, 0 }  ,draw opacity=1 ] (256.3,126.03) .. controls (256.3,68.68) and (302.79,22.19) .. (360.14,22.19) .. controls (417.49,22.19) and (463.98,68.68) .. (463.98,126.03) .. controls (463.98,183.38) and (417.49,229.87) .. (360.14,229.87) .. controls (302.79,229.87) and (256.3,183.38) .. (256.3,126.03) -- cycle;
\draw  [color={rgb, 255:red, 0; green, 0; blue, 0 }  ,draw opacity=1 ] (312.12,126.03) .. controls (312.12,99.51) and (333.62,78.01) .. (360.14,78.01) .. controls (386.66,78.01) and (408.16,99.51) .. (408.16,126.03) .. controls (408.16,152.55) and (386.66,174.05) .. (360.14,174.05) .. controls (333.62,174.05) and (312.12,152.55) .. (312.12,126.03) -- cycle ;
\draw  [color={rgb, 255:red, 0; green, 0; blue, 0 }  ,draw opacity=1 ] (291.67,126.03) .. controls (291.67,88.21) and (322.33,57.55) .. (360.14,57.55) .. controls (397.96,57.55) and (428.61,88.21) .. (428.61,126.03) .. controls (428.61,163.84) and (397.96,194.5) .. (360.14,194.5) .. controls (322.33,194.5) and (291.67,163.84) .. (291.67,126.03) -- cycle ;
\draw  [draw opacity=0][line width=1]  (401.53,150.42) .. controls (399.51,153.84) and (397.08,156.98) .. (394.31,159.79) -- (360.14,126.03) -- cycle ; \draw [color={rgb, 255:red, 189; green, 16; blue, 224 }  ,draw opacity=1 ][line width=1]    (401.53,150.42) .. controls (399.51,153.84) and (397.08,156.98) .. (394.31,159.79) ;  
\draw  [draw opacity=0][line width=1]  (313.16,175.84) .. controls (305.03,168.16) and (298.78,158.53) .. (295.16,147.68) -- (360.14,126.03) -- cycle ; \draw [color={rgb, 255:red, 189; green, 16; blue, 224 }  ,draw opacity=1 ][line width=1]    (313.16,175.84) .. controls (305.03,168.16) and (298.78,158.53) .. (295.16,147.68) ;  
\draw  [draw opacity=0][line width=1]  (392.72,27.41) .. controls (397.7,29.06) and (402.52,31.07) .. (407.14,33.42) -- (360.14,126.03) -- cycle ; \draw [color={rgb, 255:red, 189; green, 16; blue, 224 }  ,draw opacity=1 ][line width=1]    (392.72,27.41) .. controls (397.7,29.06) and (402.52,31.07) .. (407.14,33.42) ;  
\draw  [draw opacity=0][line width=1]  (353.4,229.64) .. controls (349.36,229.38) and (345.33,228.88) .. (341.32,228.15) -- (360.14,126.03) -- cycle ; \draw [color={rgb, 255:red, 189; green, 16; blue, 224 }  ,draw opacity=1 ][line width=1]    (353.4,229.64) .. controls (349.36,229.38) and (345.33,228.88) .. (341.32,228.15) ;  
\draw  [draw opacity=0][line width=1]  (372.47,229.12) .. controls (368.45,229.61) and (364.39,229.86) .. (360.32,229.87) -- (360.14,126.03) -- cycle ; \draw [color={rgb, 255:red, 189; green, 16; blue, 224 }  ,draw opacity=1 ][line width=1]    (372.47,229.12) .. controls (368.45,229.61) and (364.39,229.86) .. (360.32,229.87) ;  
\draw  [draw opacity=0][dash pattern={on 3.38pt off 3.27pt}][line width=1]  (437.76,194.99) .. controls (435.07,198.02) and (432.2,200.9) .. (429.17,203.61) -- (360.14,126.03) -- cycle ; \draw [color={rgb, 255:red, 189; green, 16; blue, 224 }  ,draw opacity=1 ][dash pattern={on 3.38pt off 3.27pt}][line width=1]  [dash pattern={on 3.38pt off 3.27pt}]  (437.76,194.99) .. controls (435.07,198.02) and (432.2,200.9) .. (429.17,203.61) ;  
\draw  [draw opacity=0][line width=1]  (257.24,139.87) .. controls (256.41,133.72) and (256.12,127.49) .. (256.4,121.24) -- (360.14,126.03) -- cycle ; \draw [color={rgb, 255:red, 189; green, 16; blue, 224 }  ,draw opacity=1 ][line width=1]    (257.24,139.87) .. controls (256.41,133.72) and (256.12,127.49) .. (256.4,121.24) ;  
\draw  [draw opacity=0][line width=1]  (326.88,66.17) .. controls (330.6,64.11) and (334.51,62.39) .. (338.56,61.04) -- (360.14,126.03) -- cycle ; \draw [color={rgb, 255:red, 189; green, 16; blue, 224 }  ,draw opacity=1 ][line width=1]    (326.88,66.17) .. controls (330.6,64.11) and (334.51,62.39) .. (338.56,61.04) ;  
\draw  [draw opacity=0][line width=1]  (322.12,155.39) .. controls (320.4,153.15) and (318.88,150.77) .. (317.57,148.28) -- (360.14,126.03) -- cycle ; \draw [color={rgb, 255:red, 189; green, 16; blue, 224 }  ,draw opacity=1 ][line width=1]    (322.12,155.39) .. controls (320.4,153.15) and (318.88,150.77) .. (317.57,148.28) ;  
\draw    (482.28,84.79) -- (462.1,96.44) ;
\draw [shift={(460.36,97.44)}, rotate = 330] [color={rgb, 255:red, 0; green, 0; blue, 0 }  ][line width=0.75]    (10.93,-3.29) .. controls (6.95,-1.4) and (3.31,-0.3) .. (0,0) .. controls (3.31,0.3) and (6.95,1.4) .. (10.93,3.29)   ;
\draw    (189.44,72.07) -- (254.15,130.9) ;
\draw [shift={(255.63,132.24)}, rotate = 222.28] [color={rgb, 255:red, 0; green, 0; blue, 0 }  ][line width=0.75]    (10.93,-3.29) .. controls (6.95,-1.4) and (3.31,-0.3) .. (0,0) .. controls (3.31,0.3) and (6.95,1.4) .. (10.93,3.29)   ;
\draw    (236.9,61.68) -- (332.27,63.95) ;
\draw [shift={(334.27,64)}, rotate = 181.36] [color={rgb, 255:red, 0; green, 0; blue, 0 }  ][line width=0.75]    (10.93,-3.29) .. controls (6.95,-1.4) and (3.31,-0.3) .. (0,0) .. controls (3.31,0.3) and (6.95,1.4) .. (10.93,3.29)   ;
\draw    (228.06,70.32) -- (299.57,159.61) ;
\draw [shift={(300.82,161.17)}, rotate = 231.31] [color={rgb, 255:red, 0; green, 0; blue, 0 }  ][line width=0.75]    (10.93,-3.29) .. controls (6.95,-1.4) and (3.31,-0.3) .. (0,0) .. controls (3.31,0.3) and (6.95,1.4) .. (10.93,3.29)   ;

\draw (483.27,68.12) node [anchor=north west][inner sep=0.75pt]    {$|z|=r_{n}$};
\draw (122.16,36.86) node [anchor=north west][inner sep=0.75pt]   [align=left] {Here the function \\ can be large};

\end{tikzpicture}

Later, I.V. Ostrovskiy significantly strengthened Keldysh's result in the work (\cite{Book_Gold_Ostr}, 6, p. 252, th. 6.1), showing that
\begin{equation*}\label{eq:Ostrov}
\forall p\in ( 0;1) \ \ \int\limits _{0}^{2\pi }\left| \sum _{n=1}^{\infty }\frac{c_{n}}{re^{i\varphi } -t_{n}}\right| ^{p} d\varphi\leqslant \frac{8\pi}{\cos\frac{\pi p}{2}}\left(\left(\sum\limits_{|t_n|>r}\frac{|c_n|}{|t_n|}\right)^p+\left(\frac{1}{r}\sum\limits_{|t_n|<r}|c_n|\right)^p\right)\xrightarrow[r\rightarrow \infty ]{} 0.
\end{equation*}
From this, thanks to the inequality $\ln^{+} a\leqslant \frac{a^{p}}{p}$, it follows that the integral from Keldysh's result \ref{eq:keld} simply tends to zero
\begin{equation*}
\lim _{r\rightarrow \infty }\int\limits _{0}^{2\pi }\ln^{+}\left| \sum _{n=1}^{\infty }\frac{c_{n}}{re^{i\varphi } -t_{n}}\right| d\varphi =0.
\end{equation*}
Now, let's again derive the pointwise version from the Chebyshev inequality
\begin{gather*}
\forall \{r_{k}\}_{k\in \mathbb{N}}  >0\ \text{strictly increasing},\ \forall \varepsilon \in ( 0;2\pi ) \ \  \exists \{E_{k}\}_{k\in \mathbb{N}} -\ \text{measurable},\ E_{k} \subset [ 0,2\pi ] \ \text{и} \ |E_{k} | >2\pi -\varepsilon \\
\text{such that} \ \ \lim _{k\rightarrow \infty } \ \sup _{\varphi \in E_{k}}\left| f\left( r_{k} e^{i\varphi }\right)\right| =0.
\end{gather*}
Looking at the figure, this means that the function uniformly tends to zero along the sequence of circles outside \textcolor[rgb]{0.56,0.07,1}{the arcs (or simply measurable sets)} of arbitrarily small angular measure.

\section{The smallness of the derivative of the sum of Cauchy kernels at infinity}

I needed similar results for working with derivatives of sums of Cauchy kernels, i.e., for functions of the form:
\begin{equation*}
\sum \limits_{n=1}^{\infty }\frac{c_{n}}{( z-t_{n})^{2}}
\end{equation*}
These are useful for studying the distribution of zeros, their growth, and most importantly, the existence by means of Nevanlinna theory. I managed to partially extend the results above, albeit with some limitations. The proof will be partially based on the original reasoning of I.V. Ostrovsky. To do this, let's present and prove a lemma by V.I. Smirnov

\begin{lema}[V.I. Smirnov, \cite{Book_Gold_Ostr}, 6, стр. 253, lemma 6.1] \label{lem:smirnova}
Let $\displaystyle f$ be analytic in the disk $\displaystyle |z|< r$ and $\displaystyle 0< p< 1$. Then
\begin{equation*}
\left[ \begin{array}{l l}
\Re f\ \text{preserves sign} \\
\Im f\ \text{preserves sign}
\end{array} \right. \Longrightarrow \ \int\limits _{0}^{2\pi } \varliminf _{ \begin{array}{l}
z\rightarrow re^{i\varphi }\\
|z|< r
\end{array}} |f( z) |^{p} d\varphi \leqslant \frac{2\pi }{\cos\frac{\pi p}{2}} |f( 0) |^{p}.
\end{equation*}
\end{lema}
\begin{proof}
Without loss of generality, let's assume that $\displaystyle \Re f >0$ in the disk $\displaystyle |z|< r$; otherwise, we multiply $\displaystyle f$ by $\displaystyle i$ or $\displaystyle -1$. Then $\displaystyle |\arg f|< \frac{\pi }{2}$ and $\displaystyle f^{p}$ is analytic in the disk $\displaystyle |z|< r$. Therefore, $\displaystyle \Re \left( f^{p}\right)$ is harmonic there, and
\begin{equation*}
\Re \left( f^{p}\right) =|f|^{p} \cdot \cos( p\arg f) \geqslant |f|^{p} \cdot \cos\frac{\pi p}{2}.
\end{equation*} 
Поэтому
\begin{gather*}
\int\limits _{0}^{2\pi } \varliminf _{ \begin{array}{l}
z\rightarrow re^{i\varphi }\\
|z|< r
\end{array}} |f( z) |^{p} d\varphi \underset{ \begin{array}{l}
\text{Lemma}\\
\text{Fatu}
\end{array}}{\leqslant }\varliminf _{ \begin{array}{l}
z\rightarrow re^{i\varphi }\\
|z|< r
\end{array}}\int\limits _{0}^{2\pi } |f( z) |^{p} d\varphi \leqslant \frac{1}{\cos\frac{\pi p}{2}}\varliminf _{ \begin{array}{l}
z\rightarrow re^{i\varphi }\\
|z|< r
\end{array}}\int\limits _{0}^{2\pi } \Re \left( f^{p}( z)\right) d\varphi =\\
\xlongequal[ \begin{array}{l}
\text{Theorem on} \\
\text{the Mean}
\end{array}]{}\frac{2\pi }{\cos\frac{\pi p}{2}}\left| \Re \left( f^{p}( 0)\right)\right| \leqslant \frac{2\pi }{\cos\frac{\pi p}{2}} |f( 0) |^{p}.
\end{gather*}
\end{proof}
Now, using Smirnov's lemma [\ref{lem:smirnova}], let's estimate the integral of the tail of our series.

\begin{lema}[The tail of the series]\label{lem:tail}
Let $\displaystyle \{c_{n}\}_{n\in \mathbb{N}} ,\{t_{n}\}_{n\in \mathbb{N}} \subset \mathbb{C}$ such that $\displaystyle \lim _{n\rightarrow \infty } |t_{n} |=\infty $ and $\displaystyle \{t_{n}\}_{n\in \mathbb{N}}$ has no finite limit points, and
\begin{equation*}
\sum _{n=1}^{\infty }\frac{|c_{n} |}{|t_{n} |^{2}} < +\infty \ -\ \text{the natural condition}.
\end{equation*}
Then for each $\displaystyle 0< p< 1$ 
\begin{equation*}
\int\limits _{0}^{2\pi }\left| \sum _{|t_{n} | >r\sqrt{2}}\frac{c_{n}}{\left( re^{i\varphi } -t_{n}\right)^{2}}\right| ^{p} d\varphi \leqslant \frac{8\pi }{\cos\frac{\pi p}{2}}\left(\sum\limits _{|t_{n} | >r\sqrt{2}}\frac{|c_{n} |}{|t_{n} |^{2}}\right)^{p}\xrightarrow[r\rightarrow \infty ]{} 0.
\end{equation*}
\end{lema}
\begin{proof}
Let's transform our sum.
\begin{gather*}
\sum _{|t_{n} | >r\sqrt{2}}\frac{c_{n}}{( z-t_{n})^{2}} =\sum _{|t_{n} | >r\sqrt{2}}\frac{e^{2i\arg t_{n}} \cdot c_{n} e^{-2i\arg t_{n}}}{( z-t_{n})^{2}} =\\
=\sum _{|t_{n} | >r\sqrt{2}}\frac{e^{2i\arg t_{n}} \cdot \Re \left( c_{n} e^{-2i\arg t_{n}}\right)}{( z-t_{n})^{2}} +i\sum _{|t_{n} | >r\sqrt{2}}\frac{e^{2i\arg t_{n}} \cdot \Im \left( c_{n} e^{-2i\arg t_{n}}\right)}{( z-t_{n})^{2}} =\\
=\sum _{ \begin{array}{l}
|t_{n} | >r\sqrt{2}\\
\Re \left( c_{n} e^{-2i\arg t_{n}}\right)  >0
\end{array}}\frac{e^{2i\arg t_{n}} \cdot \Re \left( c_{n} e^{-2i\arg t_{n}}\right)}{( z-t_{n})^{2}} +\sum _{ \begin{array}{l}
|t_{n} | >r\sqrt{2}\\
\Re \left( c_{n} e^{-2i\arg t_{n}}\right) < 0
\end{array}}\frac{e^{2i\arg t_{n}} \cdot \Re \left( c_{n} e^{-2i\arg t_{n}}\right)}{( z-t_{n})^{2}} +\\
+i\sum _{ \begin{array}{l}
|t_{n} | >r\sqrt{2}\\
\Im \left( c_{n} e^{-2i\arg t_{n}}\right)  >0
\end{array}}\frac{e^{2i\arg t_{n}} \cdot \Im \left( c_{n} e^{-2i\arg t_{n}}\right)}{( z-t_{n})^{2}} +i\sum _{ \begin{array}{l}
|t_{n} | >r\sqrt{2}\\
\Im \left( c_{n} e^{-2i\arg t_{n}}\right) < 0
\end{array}}\frac{e^{2i\arg t_{n}} \cdot \Im \left( c_{n} e^{-2i\arg t_{n}}\right)}{( z-t_{n})^{2}}.
\end{gather*}
Let's denote these four sums as $\displaystyle F_{1}( z)$, $\displaystyle F_{2}( z)$, $\displaystyle F_{3}( z)$, $\displaystyle F_{4}( z)$, then these functions are analytic in the disk $\displaystyle |z|< r$. Now, to apply Smirnov's lemma, we need to examine the sign of $\displaystyle \Re $ or $\displaystyle \Im $. To do this, let's verify that
\begin{equation*}
\Re \left(\frac{e^{2i\arg t}}{( z-t)^{2}}\right)  >0\ \text{in the disk} \ |z|< r,\ |t| >r\sqrt{2}.
\end{equation*}
For me, the simplest way to understand this is to look at the conformal map. Notice that
\begin{equation*}
\frac{e^{2i\arg t}}{( z-t)^{2}} =z^{2} \circ ze^{i\arg t} \circ \frac{1}{z} \circ ( t-z).
\end{equation*}
Since $\displaystyle |t | >r\sqrt{2}$, when shifting $\displaystyle z\rightarrow t-z$, the zero will be outside the circle, and $\displaystyle \sin \xi =\frac{r}{|t|} < \frac{1}{\sqrt{2}}$, hence $\displaystyle 0< \xi < \frac{\pi }{4}$. That's why we'll obtain $\displaystyle \Re  >0$ as a result.

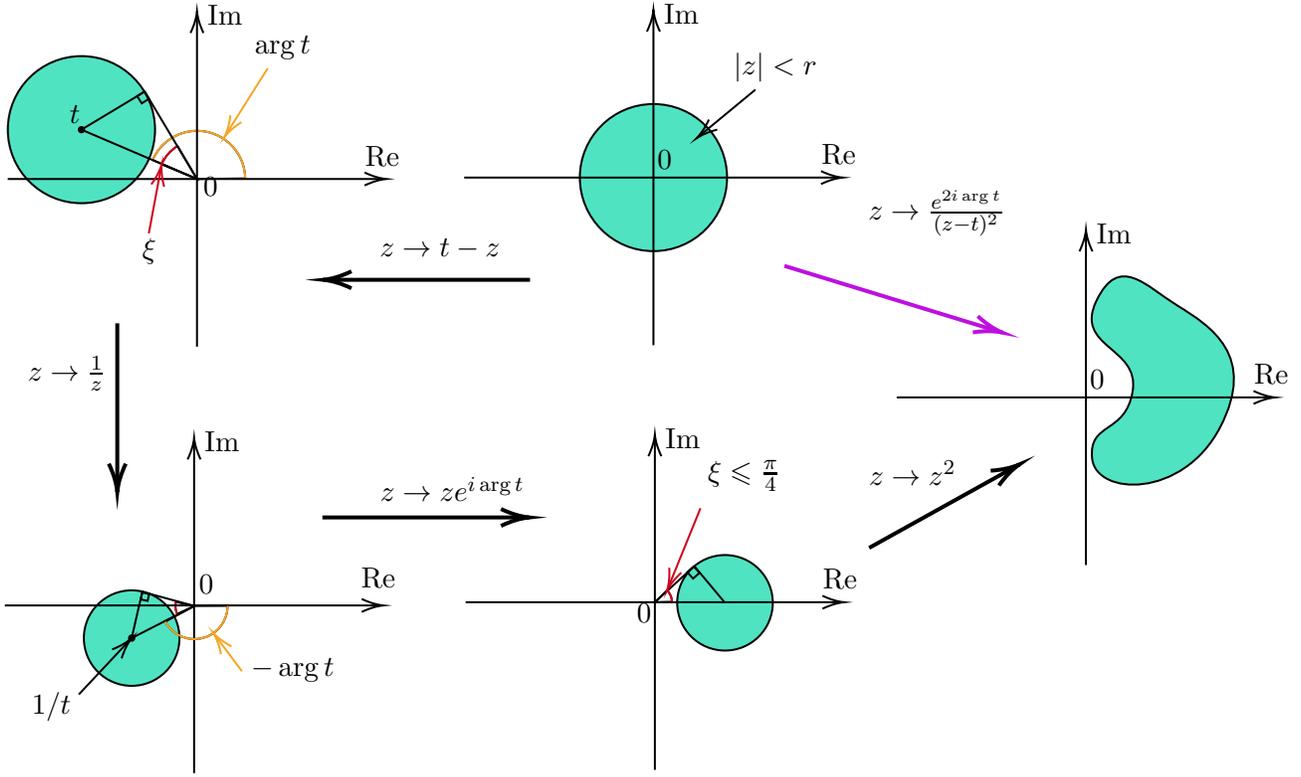
\tikzset{every picture/.style={line width=0.75pt}} 

\begin{tikzpicture}[x=0.75pt,y=0.75pt,yscale=-1,xscale=1]

\draw  [fill={rgb, 255:red, 80; green, 227; blue, 194 }  ,fill opacity=1 ] (591.85,159.45) .. controls (603.27,135.63) and (616.65,151.93) .. (632.97,162.05) .. controls (649.29,172.16) and (673.12,187.18) .. (654.19,223.4) .. controls (635.26,259.63) and (591.64,258.01) .. (589.69,239.08) .. controls (587.73,220.15) and (606.54,227.32) .. (610.13,204.96) .. controls (613.72,182.61) and (580.43,183.28) .. (591.85,159.45) -- cycle ;
\draw  [fill={rgb, 255:red, 80; green, 227; blue, 194 }  ,fill opacity=1 ] (380.41,311.89) .. controls (380.41,298.58) and (391.2,287.79) .. (404.52,287.79) .. controls (417.83,287.79) and (428.62,298.58) .. (428.62,311.89) .. controls (428.62,325.21) and (417.83,336) .. (404.52,336) .. controls (391.2,336) and (380.41,325.21) .. (380.41,311.89) -- cycle ;
\draw  [fill={rgb, 255:red, 80; green, 227; blue, 194 }  ,fill opacity=1 ] (81.17,329.71) .. controls (81.17,316.39) and (91.97,305.6) .. (105.28,305.6) .. controls (118.59,305.6) and (129.39,316.39) .. (129.39,329.71) .. controls (129.39,343.02) and (118.59,353.81) .. (105.28,353.81) .. controls (91.97,353.81) and (81.17,343.02) .. (81.17,329.71) -- cycle ;
\draw  [draw opacity=0] (120.42,90.56) .. controls (122.04,86.69) and (124.87,83.47) .. (128.45,81.36) -- (138.22,97.99) -- cycle ; \draw  [color={rgb, 255:red, 208; green, 2; blue, 27 }  ,draw opacity=1 ] (120.42,90.56) .. controls (122.04,86.69) and (124.87,83.47) .. (128.45,81.36) ;  
\draw  [draw opacity=0] (115.87,88.39) .. controls (119.6,79.73) and (128.2,73.67) .. (138.22,73.67) .. controls (151.55,73.67) and (162.37,84.4) .. (162.53,97.69) -- (138.22,97.99) -- cycle ; \draw  [color={rgb, 255:red, 245; green, 166; blue, 35 }  ,draw opacity=1 ] (115.87,88.39) .. controls (119.6,79.73) and (128.2,73.67) .. (138.22,73.67) .. controls (151.55,73.67) and (162.37,84.4) .. (162.53,97.69) ;  
\draw  [fill={rgb, 255:red, 80; green, 227; blue, 194 }  ,fill opacity=1 ] (42.82,73.09) .. controls (42.82,52.59) and (59.44,35.97) .. (79.94,35.97) .. controls (100.44,35.97) and (117.06,52.59) .. (117.06,73.09) .. controls (117.06,93.59) and (100.44,110.21) .. (79.94,110.21) .. controls (59.44,110.21) and (42.82,93.59) .. (42.82,73.09) -- cycle ;
\draw  [fill={rgb, 255:red, 80; green, 227; blue, 194 }  ,fill opacity=1 ] (331.31,97.27) .. controls (331.31,76.77) and (347.93,60.15) .. (368.43,60.15) .. controls (388.93,60.15) and (405.55,76.77) .. (405.55,97.27) .. controls (405.55,117.77) and (388.93,134.39) .. (368.43,134.39) .. controls (347.93,134.39) and (331.31,117.77) .. (331.31,97.27) -- cycle ;
\draw    (272.99,97.27) -- (461.87,97.27) ;
\draw [shift={(463.87,97.27)}, rotate = 180] [color={rgb, 255:red, 0; green, 0; blue, 0 }  ][line width=0.75]    (10.93,-3.29) .. controls (6.95,-1.4) and (3.31,-0.3) .. (0,0) .. controls (3.31,0.3) and (6.95,1.4) .. (10.93,3.29)   ;
\draw    (368.43,181.89) -- (368.43,14.65) ;
\draw [shift={(368.43,12.65)}, rotate = 90] [color={rgb, 255:red, 0; green, 0; blue, 0 }  ][line width=0.75]    (10.93,-3.29) .. controls (6.95,-1.4) and (3.31,-0.3) .. (0,0) .. controls (3.31,0.3) and (6.95,1.4) .. (10.93,3.29)   ;
\draw    (42.78,97.99) -- (231.66,97.99) ;
\draw [shift={(233.66,97.99)}, rotate = 180] [color={rgb, 255:red, 0; green, 0; blue, 0 }  ][line width=0.75]    (10.93,-3.29) .. controls (6.95,-1.4) and (3.31,-0.3) .. (0,0) .. controls (3.31,0.3) and (6.95,1.4) .. (10.93,3.29)   ;
\draw    (138.22,182.61) -- (138.22,15.37) ;
\draw [shift={(138.22,13.37)}, rotate = 90] [color={rgb, 255:red, 0; green, 0; blue, 0 }  ][line width=0.75]    (10.93,-3.29) .. controls (6.95,-1.4) and (3.31,-0.3) .. (0,0) .. controls (3.31,0.3) and (6.95,1.4) .. (10.93,3.29)   ;
\draw    (273.71,311.61) -- (462.6,311.61) ;
\draw [shift={(464.6,311.61)}, rotate = 180] [color={rgb, 255:red, 0; green, 0; blue, 0 }  ][line width=0.75]    (10.93,-3.29) .. controls (6.95,-1.4) and (3.31,-0.3) .. (0,0) .. controls (3.31,0.3) and (6.95,1.4) .. (10.93,3.29)   ;
\draw    (369.15,396.22) -- (369.15,228.99) ;
\draw [shift={(369.15,226.99)}, rotate = 90] [color={rgb, 255:red, 0; green, 0; blue, 0 }  ][line width=0.75]    (10.93,-3.29) .. controls (6.95,-1.4) and (3.31,-0.3) .. (0,0) .. controls (3.31,0.3) and (6.95,1.4) .. (10.93,3.29)   ;
\draw [color={rgb, 255:red, 0; green, 0; blue, 0 }  ,draw opacity=1 ][line width=1.5]    (306.19,148.87) -- (204.91,148.87) ;
\draw [shift={(201.91,148.87)}, rotate = 360] [color={rgb, 255:red, 0; green, 0; blue, 0 }  ,draw opacity=1 ][line width=1.5]    (14.21,-4.28) .. controls (9.04,-1.82) and (4.3,-0.39) .. (0,0) .. controls (4.3,0.39) and (9.04,1.82) .. (14.21,4.28)   ;
\draw    (79.94,73.09) -- (138.22,97.99) ;
\draw [shift={(79.94,73.09)}, rotate = 23.13] [color={rgb, 255:red, 0; green, 0; blue, 0 }  ][fill={rgb, 255:red, 0; green, 0; blue, 0 }  ][line width=0.75]      (0, 0) circle [x radius= 1.34, y radius= 1.34]   ;
\draw    (79.94,73.09) -- (111.88,53.88) ;
\draw    (111.88,53.88) -- (138.22,97.99) ;
\draw   (111.69,53.83) -- (114.01,57.48) -- (110.36,59.81) -- (108.03,56.16) -- cycle ;
\draw [color={rgb, 255:red, 208; green, 2; blue, 27 }  ,draw opacity=1 ]   (113.93,125.41) -- (120.06,92.52) ;
\draw [shift={(120.42,90.56)}, rotate = 100.54] [color={rgb, 255:red, 208; green, 2; blue, 27 }  ,draw opacity=1 ][line width=0.75]    (10.93,-3.29) .. controls (6.95,-1.4) and (3.31,-0.3) .. (0,0) .. controls (3.31,0.3) and (6.95,1.4) .. (10.93,3.29)   ;
\draw [color={rgb, 255:red, 245; green, 166; blue, 35 }  ,draw opacity=1 ]   (173.9,42.1) -- (153.1,75.12) ;
\draw [shift={(152.04,76.81)}, rotate = 302.21] [color={rgb, 255:red, 245; green, 166; blue, 35 }  ,draw opacity=1 ][line width=0.75]    (10.93,-3.29) .. controls (6.95,-1.4) and (3.31,-0.3) .. (0,0) .. controls (3.31,0.3) and (6.95,1.4) .. (10.93,3.29)   ;
\draw [color={rgb, 255:red, 0; green, 0; blue, 0 }  ,draw opacity=1 ][line width=1.5]    (98.02,170.84) -- (98.02,252.91) ;
\draw [shift={(98.02,255.91)}, rotate = 270] [color={rgb, 255:red, 0; green, 0; blue, 0 }  ,draw opacity=1 ][line width=1.5]    (14.21,-4.28) .. controls (9.04,-1.82) and (4.3,-0.39) .. (0,0) .. controls (4.3,0.39) and (9.04,1.82) .. (14.21,4.28)   ;
\draw    (105.28,329.71) -- (136.92,313.49) ;
\draw [shift={(105.28,329.71)}, rotate = 332.87] [color={rgb, 255:red, 0; green, 0; blue, 0 }  ][fill={rgb, 255:red, 0; green, 0; blue, 0 }  ][line width=0.75]      (0, 0) circle [x radius= 1.34, y radius= 1.34]   ;
\draw    (110.04,306.19) -- (136.92,313.49) ;
\draw    (105.28,329.71) -- (110.74,305.89) ;
\draw   (110.63,306.71) -- (114.09,307.48) -- (113.32,310.94) -- (109.86,310.18) -- cycle ;
\draw  [draw opacity=0] (128.24,317.51) .. controls (127.36,315.62) and (127.13,313.5) .. (127.56,311.49) -- (136.92,313.49) -- cycle ; \draw  [color={rgb, 255:red, 208; green, 2; blue, 27 }  ,draw opacity=1 ] (128.24,317.51) .. controls (127.36,315.62) and (127.13,313.5) .. (127.56,311.49) ;  
\draw    (41.33,313.23) -- (230.22,313.23) ;
\draw [shift={(232.22,313.23)}, rotate = 180] [color={rgb, 255:red, 0; green, 0; blue, 0 }  ][line width=0.75]    (10.93,-3.29) .. controls (6.95,-1.4) and (3.31,-0.3) .. (0,0) .. controls (3.31,0.3) and (6.95,1.4) .. (10.93,3.29)   ;
\draw    (136.77,397.85) -- (136.77,230.61) ;
\draw [shift={(136.77,228.61)}, rotate = 90] [color={rgb, 255:red, 0; green, 0; blue, 0 }  ][line width=0.75]    (10.93,-3.29) .. controls (6.95,-1.4) and (3.31,-0.3) .. (0,0) .. controls (3.31,0.3) and (6.95,1.4) .. (10.93,3.29)   ;
\draw    (78.64,358.15) -- (103.91,331.17) ;
\draw [shift={(105.28,329.71)}, rotate = 133.12] [color={rgb, 255:red, 0; green, 0; blue, 0 }  ][line width=0.75]    (10.93,-3.29) .. controls (6.95,-1.4) and (3.31,-0.3) .. (0,0) .. controls (3.31,0.3) and (6.95,1.4) .. (10.93,3.29)   ;
\draw  [draw opacity=0] (153.64,313.42) .. controls (153.67,319.45) and (150.41,325.29) .. (144.69,328.29) .. controls (136.58,332.55) and (126.56,329.48) .. (122.22,321.44) -- (136.92,313.49) -- cycle ; \draw  [color={rgb, 255:red, 245; green, 166; blue, 35 }  ,draw opacity=1 ] (153.64,313.42) .. controls (153.67,319.45) and (150.41,325.29) .. (144.69,328.29) .. controls (136.58,332.55) and (126.56,329.48) .. (122.22,321.44) ;  
\draw [color={rgb, 255:red, 245; green, 166; blue, 35 }  ,draw opacity=1 ]   (160.91,346.46) -- (147.61,328.69) ;
\draw [shift={(146.41,327.09)}, rotate = 53.18] [color={rgb, 255:red, 245; green, 166; blue, 35 }  ,draw opacity=1 ][line width=0.75]    (10.93,-3.29) .. controls (6.95,-1.4) and (3.31,-0.3) .. (0,0) .. controls (3.31,0.3) and (6.95,1.4) .. (10.93,3.29)   ;
\draw    (404.52,311.89) -- (389.25,293.64) ;
\draw   (388.75,293.57) -- (391.65,296.78) -- (388.44,299.68) -- (385.53,296.47) -- cycle ;
\draw    (388.81,293.57) -- (369.15,311.61) ;
\draw  [draw opacity=0] (375.33,305.59) .. controls (376.84,307.14) and (377.71,309.21) .. (377.78,311.37) -- (369.15,311.61) -- cycle ; \draw  [color={rgb, 255:red, 208; green, 2; blue, 27 }  ,draw opacity=1 ] (375.33,305.59) .. controls (376.84,307.14) and (377.71,309.21) .. (377.78,311.37) ;  
\draw [color={rgb, 255:red, 208; green, 2; blue, 27 }  ,draw opacity=1 ]   (392.14,264.19) -- (376.08,303.73) ;
\draw [shift={(375.33,305.59)}, rotate = 292.1] [color={rgb, 255:red, 208; green, 2; blue, 27 }  ,draw opacity=1 ][line width=0.75]    (10.93,-3.29) .. controls (6.95,-1.4) and (3.31,-0.3) .. (0,0) .. controls (3.31,0.3) and (6.95,1.4) .. (10.93,3.29)   ;
\draw [color={rgb, 255:red, 0; green, 0; blue, 0 }  ,draw opacity=1 ][line width=1.5]    (302.75,268.81) -- (201.47,268.81) ;
\draw [shift={(305.75,268.81)}, rotate = 180] [color={rgb, 255:red, 0; green, 0; blue, 0 }  ,draw opacity=1 ][line width=1.5]    (14.21,-4.28) .. controls (9.04,-1.82) and (4.3,-0.39) .. (0,0) .. controls (4.3,0.39) and (9.04,1.82) .. (14.21,4.28)   ;
\draw    (419.85,52.89) -- (391.62,76.24) ;
\draw [shift={(390.08,77.51)}, rotate = 320.4] [color={rgb, 255:red, 0; green, 0; blue, 0 }  ][line width=0.75]    (10.93,-3.29) .. controls (6.95,-1.4) and (3.31,-0.3) .. (0,0) .. controls (3.31,0.3) and (6.95,1.4) .. (10.93,3.29)   ;
\draw    (491.12,208.23) -- (680,208.23) ;
\draw [shift={(682,208.23)}, rotate = 180] [color={rgb, 255:red, 0; green, 0; blue, 0 }  ][line width=0.75]    (10.93,-3.29) .. controls (6.95,-1.4) and (3.31,-0.3) .. (0,0) .. controls (3.31,0.3) and (6.95,1.4) .. (10.93,3.29)   ;
\draw    (586.56,292.84) -- (586.56,125.61) ;
\draw [shift={(586.56,123.61)}, rotate = 90] [color={rgb, 255:red, 0; green, 0; blue, 0 }  ][line width=0.75]    (10.93,-3.29) .. controls (6.95,-1.4) and (3.31,-0.3) .. (0,0) .. controls (3.31,0.3) and (6.95,1.4) .. (10.93,3.29)   ;
\draw [color={rgb, 255:red, 0; green, 0; blue, 0 }  ,draw opacity=1 ][line width=1.5]    (477.22,284.18) -- (550.38,243.51) ;
\draw [shift={(553,242.05)}, rotate = 150.93] [color={rgb, 255:red, 0; green, 0; blue, 0 }  ,draw opacity=1 ][line width=1.5]    (14.21,-4.28) .. controls (9.04,-1.82) and (4.3,-0.39) .. (0,0) .. controls (4.3,0.39) and (9.04,1.82) .. (14.21,4.28)   ;
\draw [color={rgb, 255:red, 189; green, 16; blue, 224 }  ,draw opacity=1 ][line width=1.5]    (434.47,141.83) -- (541.47,174.6) ;
\draw [shift={(544.34,175.48)}, rotate = 197.03] [color={rgb, 255:red, 189; green, 16; blue, 224 }  ,draw opacity=1 ][line width=1.5]    (14.21,-4.28) .. controls (9.04,-1.82) and (4.3,-0.39) .. (0,0) .. controls (4.3,0.39) and (9.04,1.82) .. (14.21,4.28)   ;

\draw (451.74,78.97) node [anchor=north west][inner sep=0.75pt]    {$\Re $};
\draw (372.31,7.88) node [anchor=north west][inner sep=0.75pt]    {$\Im $};
\draw (221.53,79.69) node [anchor=north west][inner sep=0.75pt]    {$\Re $};
\draw (142.1,8.6) node [anchor=north west][inner sep=0.75pt]    {$\Im $};
\draw (452.46,293.3) node [anchor=north west][inner sep=0.75pt]    {$\Re $};
\draw (373.04,222.22) node [anchor=north west][inner sep=0.75pt]    {$\Im $};
\draw (228.99,126.6) node [anchor=north west][inner sep=0.75pt]    {$z\rightarrow t-z$};
\draw (72.56,59.52) node [anchor=north west][inner sep=0.75pt]    {$t$};
\draw (109.16,126.31) node [anchor=north west][inner sep=0.75pt]    {$\xi $};
\draw (368.95,82.14) node [anchor=north west][inner sep=0.75pt]    {$0$};
\draw (165.83,24.38) node [anchor=north west][inner sep=0.75pt]    {$\arg t$};
\draw (51.42,185.64) node [anchor=north west][inner sep=0.75pt]    {$z\rightarrow \frac{1}{z}$};
\draw (53.43,355.31) node [anchor=north west][inner sep=0.75pt]    {$1/t$};
\draw (219.54,292.76) node [anchor=north west][inner sep=0.75pt]    {$\Re $};
\draw (140.66,223.84) node [anchor=north west][inner sep=0.75pt]    {$\Im $};
\draw (164,338.48) node [anchor=north west][inner sep=0.75pt]    {$-\arg t$};
\draw (394.41,238.46) node [anchor=north west][inner sep=0.75pt]    {$\xi \leqslant \frac{\pi }{4}$};
\draw (229.05,246.66) node [anchor=north west][inner sep=0.75pt]    {$z\rightarrow ze^{i\arg t}$};
\draw (407.43,32.78) node [anchor=north west][inner sep=0.75pt]    {$|z|< r$};
\draw (669.87,189.92) node [anchor=north west][inner sep=0.75pt]    {$\Re $};
\draw (590.44,118.84) node [anchor=north west][inner sep=0.75pt]    {$\Im $};
\draw (587.08,193.1) node [anchor=north west][inner sep=0.75pt]    {$0$};
\draw (140,95.67) node [anchor=north west][inner sep=0.75pt]    {$0$};
\draw (137.84,296.48) node [anchor=north west][inner sep=0.75pt]    {$0$};
\draw (358.67,310.55) node [anchor=north west][inner sep=0.75pt]    {$0$};
\draw (475.67,238.58) node [anchor=north west][inner sep=0.75pt]    {$z\rightarrow z^{2}$};
\draw (475.32,102.3) node [anchor=north west][inner sep=0.75pt]    {$z\rightarrow \frac{e^{2i\arg t}}{( z-t)^{2}}$};

\end{tikzpicture}

Now the sign of $\displaystyle \Re $ or $\displaystyle \Im $ of the functions is clear. $\displaystyle F_{1}( z) ,\ F_{2}( z) ,\ F_{3}( z) ,\ F_{4}( z)$:
\begin{equation*}
\Re F_{1}( z) =\sum _{ \begin{array}{l}
|t_{n} | >r\sqrt{2}\\
\Re \left( c_{n} e^{-2i\arg t_{n}}\right)  >0
\end{array}}\underbrace{\Re \left( c_{n} e^{-2i\arg t_{n}}\right)}_{ >0}\underbrace{\Re \left(\frac{e^{2i\arg t_{n}}}{( z-t_{n})^{2}}\right)}_{ >0}  >0,
\end{equation*}
\begin{equation*}
\Re F_{2}( z) =\sum _{ \begin{array}{l}
|t_{n} | >r\sqrt{2}\\
\Re \left( c_{n} e^{-2i\arg t_{n}}\right) < 0
\end{array}}\underbrace{\Re \left( c_{n} e^{-2i\arg t_{n}}\right)}_{< 0}\underbrace{\Re \left(\frac{e^{2i\arg t_{n}}}{( z-t_{n})^{2}}\right)}_{ >0} < 0,
\end{equation*}
\begin{equation*}
\Im F_{3}( z) =\sum _{ \begin{array}{l}
|t_{n} | >r\sqrt{2}\\
\Im \left( c_{n} e^{-2i\arg t_{n}}\right)  >0
\end{array}}\underbrace{\Im \left( c_{n} e^{-2i\arg t_{n}}\right)}_{ >0}\underbrace{\Re \left(\frac{e^{2i\arg t_{n}}}{( z-t_{n})^{2}}\right)}_{ >0}  >0,
\end{equation*}
\begin{equation*}
\Im F_{4}( z) =\sum _{ \begin{array}{l}
|t_{n} | >r\sqrt{2}\\
\Im \left( c_{n} e^{-2i\arg t_{n}}\right) < 0
\end{array}}\underbrace{\Im \left( c_{n} e^{-2i\arg t_{n}}\right)}_{< 0}\underbrace{\Re \left(\frac{e^{2i\arg t_{n}}}{( z-t_{n})^{2}}\right)}_{ >0} < 0.
\end{equation*}
Therefore, by Smirnov's lemma \ref{lem:smirnova}, the estimate holds.
\begin{gather*}
\int\limits _{0}^{2\pi }\left| F_{1}\left( re^{i\varphi }\right)\right| ^{p} d\varphi \leqslant \frac{2\pi }{\cos\frac{\pi p}{2}} |F_{1}( 0) |^{p} \leqslant \frac{2\pi }{\cos\frac{\pi p}{2}}\left(\sum _{ \begin{array}{l}
|t_{n} | >r\sqrt{2}\\
\Re \left( c_{n} e^{-2i\arg t_{n}}\right)  >0
\end{array}}\frac{\left| e^{2i\arg t_{n}}\right| \cdot \overbrace{\left| \Re \left( c_{n} e^{-2i\arg t_{n}}\right)\right| }^{\leqslant |c_{n} |}}{|t_{n} |^{2}}\right)^{p} \leqslant \\
\leqslant \frac{2\pi }{\cos\frac{\pi p}{2}}\left(\sum _{ \begin{array}{l}
|t_{n} | >r\sqrt{2}\\
\Re \left( c_{n} e^{-2i\arg t_{n}}\right)  >0
\end{array}}\frac{|c_{n} |}{|t_{n} |^{2}}\right)^{p} \leqslant \frac{2\pi }{\cos\frac{\pi p}{2}}\left(\sum _{|t_{n} | >r\sqrt{2}}\frac{|c_{n} |}{|t_{n} |^{2}}\right)^{p}.
\end{gather*}
Similarly, we obtain estimates.
\begin{equation*}
\int\limits _{0}^{2\pi }\left| F_{j}\left( re^{i\varphi }\right)\right| ^{p} d\varphi \leqslant \frac{2\pi }{\cos\frac{\pi p}{2}}\left(\sum _{|t_{n} | >r\sqrt{2}}\frac{|c_{n} |}{|t_{n} |^{2}}\right)^{p} ,\ j=2,3,4.
\end{equation*}
Now, let's estimate the entire integral using the inequality $\displaystyle |x+y|^{p} \leqslant |x|^{p} +|y|^{p}$ при $\displaystyle 0< p< 1$:
\begin{gather*}
\int\limits _{0}^{2\pi }\left| \sum _{|t_{n} | >r\sqrt{2}}\frac{c_{n}}{\left( re^{i\varphi } -t_{n}\right)^{2}}\right| ^{p} d\varphi =\int\limits _{0}^{2\pi }\left| \sum _{j=1}^{4} F_{j}\left( re^{i\varphi }\right)\right| ^{p} d\varphi \leqslant \\
\leqslant \sum _{j=1}^{4}\int\limits _{0}^{2\pi }\left| F_{j}\left( re^{i\varphi }\right)\right| ^{p} d\varphi \leqslant \frac{8\pi }{\cos\frac{\pi p}{2}}\left(\underbrace{\sum\limits _{|t_{n} | >r\sqrt{2}}\frac{|c_{n} |}{|t_{n} |^{2}}}_{=o( 1)}\right)^{p}\xrightarrow[r\rightarrow \infty ]{} 0.
\end{gather*}
\end{proof}
Now, in a similar manner, let's estimate the integral from the beginning of the series.

\begin{lema}[The beginning of the series]\label{lem:start}
Let $\displaystyle \{c_{n}\}_{n\in \mathbb{N}} ,\{t_{n}\}_{n\in \mathbb{N}} \subset \mathbb{C}$ such that $\displaystyle \lim _{n\rightarrow \infty } |t_{n} |=\infty $ and $\displaystyle \{t_{n}\}_{n\in \mathbb{N}}$ has no finite limit points, and
\begin{equation*}
\sum _{n=1}^{\infty }\frac{|c_{n} |}{|t_{n} |^{2}} < +\infty \ -\ \text{the natural condition}.
\end{equation*}
Then for each $\displaystyle 0< p< 1$ 
\begin{equation*}
\int\limits _{0}^{2\pi }\left| \sum _{|t_{n} |< \frac{r}{\sqrt{2}}}\frac{c_{n}}{\left( re^{i\varphi } -t_{n}\right)^{2}}\right| ^{p} d\varphi \leqslant \frac{8\pi }{\cos\frac{\pi p}{2}}\left(\frac{1}{r^{2}}\sum\limits _{|t_{n} |< \frac{r}{\sqrt{2}}} |c_{n} |\right)^{p}\xrightarrow[r\rightarrow \infty ]{} 0.
\end{equation*}
\end{lema}
\begin{proof}
The function $\displaystyle \sum _{|t_{n} |< \frac{r}{\sqrt{2}}}\frac{c_{n}}{( z-t_{n})^{2}}$ has poles inside the circle $\displaystyle |z|<r$, which prevents us from directly applying the estimation technique via Smirnov's lemma. However, on the integration contour $\displaystyle |z|=r$, we can modify the function while preserving its magnitude to reflect the poles symmetrically outside the circle $\displaystyle |z|< r$. Specifically,
\begin{equation*}
\left| \sum _{|t_{n} |< \frac{r}{\sqrt{2}}}\frac{c_{n}}{( z-t_{n})^{2}}\right| =\left| \overline{\sum\limits _{|t_{n} |< \frac{r}{\sqrt{2}}}\frac{c_{n}}{( z-t_{n})^{2}}}\right| =\left| \sum _{|t_{n} |< \frac{r}{\sqrt{2}}}\frac{\overline{c_{n}}}{\left(\overline{z} -\overline{t_{n}}\right)^{2}}\right| =\left| \sum _{|t_{n} |< \frac{r}{\sqrt{2}}}\frac{\overline{c_{n}}}{\left(\overline{z} z-z\overline{t_{n}}\right)^{2}}\right| =\left| \sum _{|t_{n} |< \frac{r}{\sqrt{2}}}\frac{\overline{c_{n}} r^{2}}{\left( r^{2} -z\overline{t_{n}}\right)^{2}}\right|.
\end{equation*}
The new function $\displaystyle \sum _{|t_{n} |< \frac{r}{\sqrt{2}}}\frac{\overline{c_{n}} r^{2}}{\left( r^{2} -z\overline{t_{n}}\right)^{2}}$ already has poles at $\displaystyle \frac{r^{2}}{\overline{t_{n}}}$, which are now outside the circle $\displaystyle |z|< r$. Now let's transform it as before:
\begin{gather*}
\sum _{|t_{n} |< \frac{r}{\sqrt{2}}}\frac{\overline{c_{n}} r^{2}}{\left( r^{2} -z\overline{t_{n}}\right)^{2}} =\sum _{|t_{n} |< \frac{r}{\sqrt{2}}}\frac{r^{2} \Re c_{n}}{\left( r^{2} -z\overline{t_{n}}\right)^{2}} -i\sum _{|t_{n} |< \frac{r}{\sqrt{2}}}\frac{r^{2} \Im c_{n}}{\left( r^{2} -z\overline{t_{n}}\right)^{2}} =\\
=\sum _{ \begin{array}{l}
|t_{n} |< \frac{r}{\sqrt{2}}\\
\Re c_{n}  >0
\end{array}}\frac{r^{2} \Re c_{n}}{\left( r^{2} -z\overline{t_{n}}\right)^{2}} +\sum _{ \begin{array}{l}
|t_{n} |< \frac{r}{\sqrt{2}}\\
\Re c_{n} < 0
\end{array}}\frac{r^{2} \Re c_{n}}{\left( r^{2} -z\overline{t_{n}}\right)^{2}} -i\sum _{ \begin{array}{l}
|t_{n} |< \frac{r}{\sqrt{2}}\\
\Im c_{n}  >0
\end{array}}\frac{r^{2} \Im c_{n}}{\left( r^{2} -z\overline{t_{n}}\right)^{2}} -i\sum _{ \begin{array}{l}
|t_{n} |< \frac{r}{\sqrt{2}}\\
\Im c_{n} < 0
\end{array}}\frac{r^{2} \Im c_{n}}{\left( r^{2} -z\overline{t_{n}}\right)^{2}}.
\end{gather*}
Let's denote these four sums as $\displaystyle G_{1}( z)$, $\displaystyle G_{2}( z)$, $\displaystyle G_{3}( z)$, $\displaystyle G_{4}( z)$. Then these functions are analytic in the disk $\displaystyle |z|< r$. Now, to apply Smirnov's lemma, we need to examine the sign of $\displaystyle \Re $ or $\displaystyle \Im $. To do this, let's verify that
\begin{equation*}
\Re \left(\frac{1}{\left( r^{2} -z\overline{t}\right)^{2}}\right)  >0\ \text{in the disc} \ |z|< r,\ |t|< \frac{r}{\sqrt{2}}.
\end{equation*}
In my opinion, the simplest way to understand this is to look at the conformal map. Note that
\begin{equation*}
\frac{1}{\left( r^{2} -z\overline{t}\right)^{2}} =z^{2} \circ \frac{1}{z} \circ \left( r^{2} -z\right) \circ ( z\overline{t}).
\end{equation*}
Since according to the condition $\displaystyle |t |< \frac{r}{\sqrt{2}}$, then $\displaystyle \sin \xi =\frac{r|t|}{r^{2}} =\frac{|t|}{r} < \frac{1}{\sqrt{2}}$, hence $\displaystyle 0< \xi < \frac{\pi }{4}$. That's why we'll obtain $\displaystyle \Re  >0$ as a result.

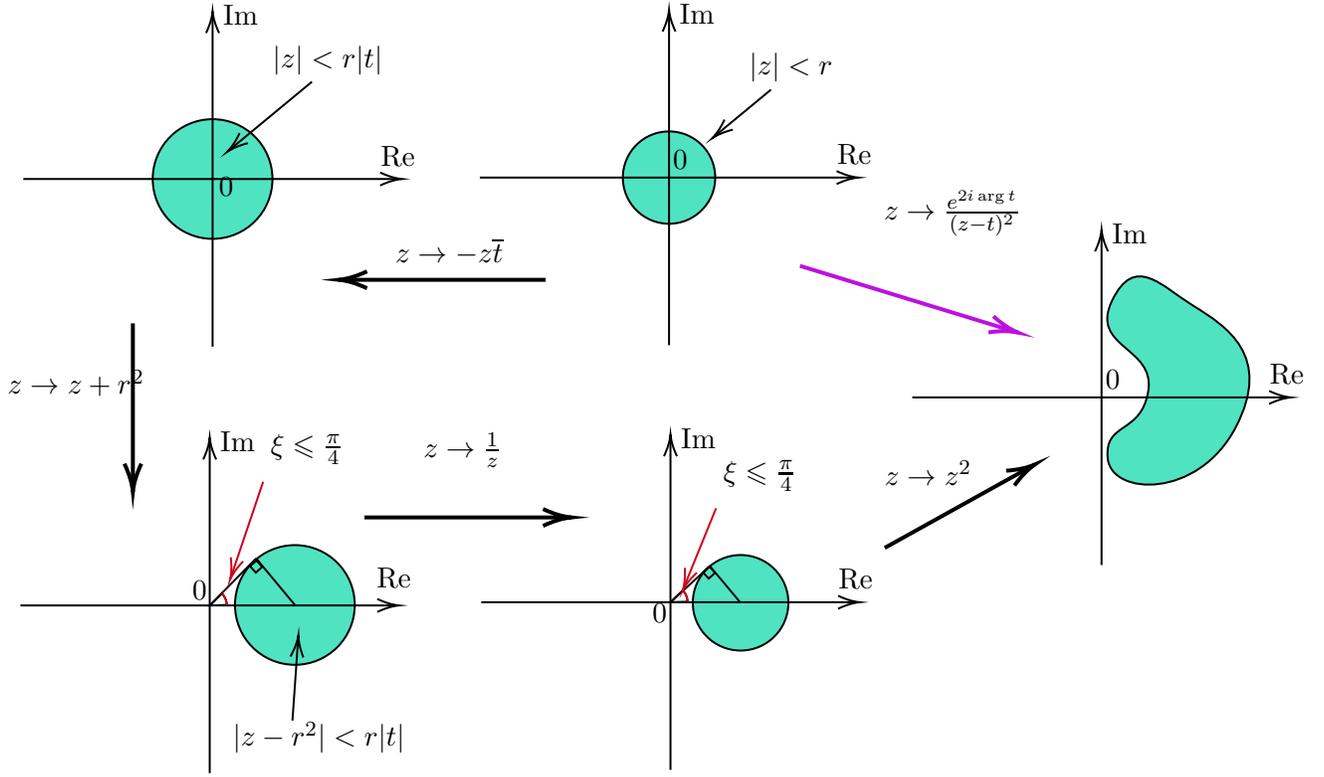
\tikzset{every picture/.style={line width=0.75pt}} 

\begin{tikzpicture}[x=0.75pt,y=0.75pt,yscale=-1,xscale=1]

\draw  [fill={rgb, 255:red, 80; green, 227; blue, 194 }  ,fill opacity=1 ] (149.52,312.99) .. controls (149.52,296.31) and (163.04,282.79) .. (179.72,282.79) .. controls (196.4,282.79) and (209.92,296.31) .. (209.92,312.99) .. controls (209.92,329.67) and (196.4,343.19) .. (179.72,343.19) .. controls (163.04,343.19) and (149.52,329.67) .. (149.52,312.99) -- cycle ;
\draw  [fill={rgb, 255:red, 80; green, 227; blue, 194 }  ,fill opacity=1 ] (591.85,159.45) .. controls (603.27,135.63) and (616.65,151.93) .. (632.97,162.05) .. controls (649.29,172.16) and (673.12,187.18) .. (654.19,223.4) .. controls (635.26,259.63) and (591.64,258.01) .. (589.69,239.08) .. controls (587.73,220.15) and (606.54,227.32) .. (610.13,204.96) .. controls (613.72,182.61) and (580.43,183.28) .. (591.85,159.45) -- cycle ;
\draw  [fill={rgb, 255:red, 80; green, 227; blue, 194 }  ,fill opacity=1 ] (380.41,311.89) .. controls (380.41,298.58) and (391.2,287.79) .. (404.52,287.79) .. controls (417.83,287.79) and (428.62,298.58) .. (428.62,311.89) .. controls (428.62,325.21) and (417.83,336) .. (404.52,336) .. controls (391.2,336) and (380.41,325.21) .. (380.41,311.89) -- cycle ;
\draw  [fill={rgb, 255:red, 80; green, 227; blue, 194 }  ,fill opacity=1 ] (108.02,97.99) .. controls (108.02,81.31) and (121.54,67.79) .. (138.22,67.79) .. controls (154.9,67.79) and (168.42,81.31) .. (168.42,97.99) .. controls (168.42,114.67) and (154.9,128.19) .. (138.22,128.19) .. controls (121.54,128.19) and (108.02,114.67) .. (108.02,97.99) -- cycle ;
\draw  [fill={rgb, 255:red, 80; green, 227; blue, 194 }  ,fill opacity=1 ] (345.13,97.27) .. controls (345.13,84.4) and (355.56,73.97) .. (368.43,73.97) .. controls (381.3,73.97) and (391.73,84.4) .. (391.73,97.27) .. controls (391.73,110.14) and (381.3,120.57) .. (368.43,120.57) .. controls (355.56,120.57) and (345.13,110.14) .. (345.13,97.27) -- cycle ;
\draw    (272.99,97.27) -- (461.87,97.27) ;
\draw [shift={(463.87,97.27)}, rotate = 180] [color={rgb, 255:red, 0; green, 0; blue, 0 }  ][line width=0.75]    (10.93,-3.29) .. controls (6.95,-1.4) and (3.31,-0.3) .. (0,0) .. controls (3.31,0.3) and (6.95,1.4) .. (10.93,3.29)   ;
\draw    (368.43,181.89) -- (368.43,14.65) ;
\draw [shift={(368.43,12.65)}, rotate = 90] [color={rgb, 255:red, 0; green, 0; blue, 0 }  ][line width=0.75]    (10.93,-3.29) .. controls (6.95,-1.4) and (3.31,-0.3) .. (0,0) .. controls (3.31,0.3) and (6.95,1.4) .. (10.93,3.29)   ;
\draw    (42.78,97.99) -- (231.66,97.99) ;
\draw [shift={(233.66,97.99)}, rotate = 180] [color={rgb, 255:red, 0; green, 0; blue, 0 }  ][line width=0.75]    (10.93,-3.29) .. controls (6.95,-1.4) and (3.31,-0.3) .. (0,0) .. controls (3.31,0.3) and (6.95,1.4) .. (10.93,3.29)   ;
\draw    (138.22,182.61) -- (138.22,15.37) ;
\draw [shift={(138.22,13.37)}, rotate = 90] [color={rgb, 255:red, 0; green, 0; blue, 0 }  ][line width=0.75]    (10.93,-3.29) .. controls (6.95,-1.4) and (3.31,-0.3) .. (0,0) .. controls (3.31,0.3) and (6.95,1.4) .. (10.93,3.29)   ;
\draw    (273.71,311.61) -- (462.6,311.61) ;
\draw [shift={(464.6,311.61)}, rotate = 180] [color={rgb, 255:red, 0; green, 0; blue, 0 }  ][line width=0.75]    (10.93,-3.29) .. controls (6.95,-1.4) and (3.31,-0.3) .. (0,0) .. controls (3.31,0.3) and (6.95,1.4) .. (10.93,3.29)   ;
\draw    (369.15,396.22) -- (369.15,228.99) ;
\draw [shift={(369.15,226.99)}, rotate = 90] [color={rgb, 255:red, 0; green, 0; blue, 0 }  ][line width=0.75]    (10.93,-3.29) .. controls (6.95,-1.4) and (3.31,-0.3) .. (0,0) .. controls (3.31,0.3) and (6.95,1.4) .. (10.93,3.29)   ;
\draw [color={rgb, 255:red, 0; green, 0; blue, 0 }  ,draw opacity=1 ][line width=1.5]    (306.19,148.87) -- (204.91,148.87) ;
\draw [shift={(201.91,148.87)}, rotate = 360] [color={rgb, 255:red, 0; green, 0; blue, 0 }  ,draw opacity=1 ][line width=1.5]    (14.21,-4.28) .. controls (9.04,-1.82) and (4.3,-0.39) .. (0,0) .. controls (4.3,0.39) and (9.04,1.82) .. (14.21,4.28)   ;
\draw [color={rgb, 255:red, 0; green, 0; blue, 0 }  ,draw opacity=1 ][line width=1.5]    (98.02,170.84) -- (98.02,252.91) ;
\draw [shift={(98.02,255.91)}, rotate = 270] [color={rgb, 255:red, 0; green, 0; blue, 0 }  ,draw opacity=1 ][line width=1.5]    (14.21,-4.28) .. controls (9.04,-1.82) and (4.3,-0.39) .. (0,0) .. controls (4.3,0.39) and (9.04,1.82) .. (14.21,4.28)   ;
\draw    (41.33,313.23) -- (230.22,313.23) ;
\draw [shift={(232.22,313.23)}, rotate = 180] [color={rgb, 255:red, 0; green, 0; blue, 0 }  ][line width=0.75]    (10.93,-3.29) .. controls (6.95,-1.4) and (3.31,-0.3) .. (0,0) .. controls (3.31,0.3) and (6.95,1.4) .. (10.93,3.29)   ;
\draw    (136.77,397.85) -- (136.77,230.61) ;
\draw [shift={(136.77,228.61)}, rotate = 90] [color={rgb, 255:red, 0; green, 0; blue, 0 }  ][line width=0.75]    (10.93,-3.29) .. controls (6.95,-1.4) and (3.31,-0.3) .. (0,0) .. controls (3.31,0.3) and (6.95,1.4) .. (10.93,3.29)   ;
\draw    (404.52,311.89) -- (389.25,293.64) ;
\draw   (388.75,293.57) -- (391.65,296.78) -- (388.44,299.68) -- (385.53,296.47) -- cycle ;
\draw    (388.81,293.57) -- (369.15,311.61) ;
\draw  [draw opacity=0] (375.33,305.59) .. controls (376.84,307.14) and (377.71,309.21) .. (377.78,311.37) -- (369.15,311.61) -- cycle ; \draw  [color={rgb, 255:red, 208; green, 2; blue, 27 }  ,draw opacity=1 ] (375.33,305.59) .. controls (376.84,307.14) and (377.71,309.21) .. (377.78,311.37) ;  
\draw [color={rgb, 255:red, 208; green, 2; blue, 27 }  ,draw opacity=1 ]   (392.14,264.19) -- (376.08,303.73) ;
\draw [shift={(375.33,305.59)}, rotate = 292.1] [color={rgb, 255:red, 208; green, 2; blue, 27 }  ,draw opacity=1 ][line width=0.75]    (10.93,-3.29) .. controls (6.95,-1.4) and (3.31,-0.3) .. (0,0) .. controls (3.31,0.3) and (6.95,1.4) .. (10.93,3.29)   ;
\draw [color={rgb, 255:red, 0; green, 0; blue, 0 }  ,draw opacity=1 ][line width=1.5]    (316.09,268.81) -- (214.81,268.81) ;
\draw [shift={(319.09,268.81)}, rotate = 180] [color={rgb, 255:red, 0; green, 0; blue, 0 }  ,draw opacity=1 ][line width=1.5]    (14.21,-4.28) .. controls (9.04,-1.82) and (4.3,-0.39) .. (0,0) .. controls (4.3,0.39) and (9.04,1.82) .. (14.21,4.28)   ;
\draw    (419.85,52.89) -- (391.62,76.24) ;
\draw [shift={(390.08,77.51)}, rotate = 320.4] [color={rgb, 255:red, 0; green, 0; blue, 0 }  ][line width=0.75]    (10.93,-3.29) .. controls (6.95,-1.4) and (3.31,-0.3) .. (0,0) .. controls (3.31,0.3) and (6.95,1.4) .. (10.93,3.29)   ;
\draw    (491.12,208.23) -- (680,208.23) ;
\draw [shift={(682,208.23)}, rotate = 180] [color={rgb, 255:red, 0; green, 0; blue, 0 }  ][line width=0.75]    (10.93,-3.29) .. controls (6.95,-1.4) and (3.31,-0.3) .. (0,0) .. controls (3.31,0.3) and (6.95,1.4) .. (10.93,3.29)   ;
\draw    (586.56,292.84) -- (586.56,125.61) ;
\draw [shift={(586.56,123.61)}, rotate = 90] [color={rgb, 255:red, 0; green, 0; blue, 0 }  ][line width=0.75]    (10.93,-3.29) .. controls (6.95,-1.4) and (3.31,-0.3) .. (0,0) .. controls (3.31,0.3) and (6.95,1.4) .. (10.93,3.29)   ;
\draw [color={rgb, 255:red, 0; green, 0; blue, 0 }  ,draw opacity=1 ][line width=1.5]    (477.22,284.18) -- (550.38,243.51) ;
\draw [shift={(553,242.05)}, rotate = 150.93] [color={rgb, 255:red, 0; green, 0; blue, 0 }  ,draw opacity=1 ][line width=1.5]    (14.21,-4.28) .. controls (9.04,-1.82) and (4.3,-0.39) .. (0,0) .. controls (4.3,0.39) and (9.04,1.82) .. (14.21,4.28)   ;
\draw [color={rgb, 255:red, 189; green, 16; blue, 224 }  ,draw opacity=1 ][line width=1.5]    (434.47,141.83) -- (541.47,174.6) ;
\draw [shift={(544.34,175.48)}, rotate = 197.03] [color={rgb, 255:red, 189; green, 16; blue, 224 }  ,draw opacity=1 ][line width=1.5]    (14.21,-4.28) .. controls (9.04,-1.82) and (4.3,-0.39) .. (0,0) .. controls (4.3,0.39) and (9.04,1.82) .. (14.21,4.28)   ;
\draw    (179.72,312.99) -- (160.2,289.5) ;
\draw    (160.2,289.5) -- (136.77,313.23) ;
\draw   (160.35,290.57) -- (163.25,293.78) -- (160.04,296.68) -- (157.13,293.47) -- cycle ;
\draw [color={rgb, 255:red, 208; green, 2; blue, 27 }  ,draw opacity=1 ]   (163.73,250.79) -- (147.64,298.41) ;
\draw [shift={(147,300.3)}, rotate = 288.67] [color={rgb, 255:red, 208; green, 2; blue, 27 }  ,draw opacity=1 ][line width=0.75]    (10.93,-3.29) .. controls (6.95,-1.4) and (3.31,-0.3) .. (0,0) .. controls (3.31,0.3) and (6.95,1.4) .. (10.93,3.29)   ;
\draw  [draw opacity=0] (142.95,307.21) .. controls (144.46,308.76) and (145.34,310.83) .. (145.4,312.99) -- (136.77,313.23) -- cycle ; \draw  [color={rgb, 255:red, 208; green, 2; blue, 27 }  ,draw opacity=1 ] (142.95,307.21) .. controls (144.46,308.76) and (145.34,310.83) .. (145.4,312.99) ;  
\draw    (188.35,48.89) -- (147.04,83.08) ;
\draw [shift={(145.5,84.35)}, rotate = 320.39] [color={rgb, 255:red, 0; green, 0; blue, 0 }  ][line width=0.75]    (10.93,-3.29) .. controls (6.95,-1.4) and (3.31,-0.3) .. (0,0) .. controls (3.31,0.3) and (6.95,1.4) .. (10.93,3.29)   ;
\draw    (178.5,371.35) -- (181.36,330.35) ;
\draw [shift={(181.5,328.35)}, rotate = 93.99] [color={rgb, 255:red, 0; green, 0; blue, 0 }  ][line width=0.75]    (10.93,-3.29) .. controls (6.95,-1.4) and (3.31,-0.3) .. (0,0) .. controls (3.31,0.3) and (6.95,1.4) .. (10.93,3.29)   ;

\draw (451.74,78.97) node [anchor=north west][inner sep=0.75pt]    {$\Re $};
\draw (372.31,7.88) node [anchor=north west][inner sep=0.75pt]    {$\Im $};
\draw (221.53,79.69) node [anchor=north west][inner sep=0.75pt]    {$\Re $};
\draw (142.1,8.6) node [anchor=north west][inner sep=0.75pt]    {$\Im $};
\draw (452.46,293.3) node [anchor=north west][inner sep=0.75pt]    {$\Re $};
\draw (373.04,222.22) node [anchor=north west][inner sep=0.75pt]    {$\Im $};
\draw (228.99,126.6) node [anchor=north west][inner sep=0.75pt]    {$z\rightarrow -z\overline{t}$};
\draw (368.95,82.14) node [anchor=north west][inner sep=0.75pt]    {$0$};
\draw (33.42,192.64) node [anchor=north west][inner sep=0.75pt]    {$z\rightarrow z+r^{2}$};
\draw (219.54,292.76) node [anchor=north west][inner sep=0.75pt]    {$\Re $};
\draw (140.66,223.84) node [anchor=north west][inner sep=0.75pt]    {$\Im $};
\draw (394.41,238.46) node [anchor=north west][inner sep=0.75pt]    {$\xi \leqslant \frac{\pi }{4}$};
\draw (243.14,224.41) node [anchor=north west][inner sep=0.75pt]    {$z\rightarrow \frac{1}{z}$};
\draw (407.43,32.78) node [anchor=north west][inner sep=0.75pt]    {$|z|< r$};
\draw (669.87,189.92) node [anchor=north west][inner sep=0.75pt]    {$\Re $};
\draw (590.44,118.84) node [anchor=north west][inner sep=0.75pt]    {$\Im $};
\draw (587.08,193.1) node [anchor=north west][inner sep=0.75pt]    {$0$};
\draw (140,95.67) node [anchor=north west][inner sep=0.75pt]    {$0$};
\draw (126.64,299.28) node [anchor=north west][inner sep=0.75pt]    {$0$};
\draw (358.67,310.55) node [anchor=north west][inner sep=0.75pt]    {$0$};
\draw (475.67,238.58) node [anchor=north west][inner sep=0.75pt]    {$z\rightarrow z^{2}$};
\draw (475.32,102.3) node [anchor=north west][inner sep=0.75pt]    {$z\rightarrow \frac{e^{2i\arg t}}{( z-t)^{2}}$};
\draw (166,225.06) node [anchor=north west][inner sep=0.75pt]    {$\xi \leqslant \frac{\pi }{4}$};
\draw (147,370.4) node [anchor=north west][inner sep=0.75pt]    {$|z-r^{2} |< r|t|$};
\draw (167,29) node [anchor=north west][inner sep=0.75pt]    {$|z|< r|t|$};

\end{tikzpicture}

Now the sign of $\displaystyle \Re $ or $\displaystyle \Im $ of the functions is clear. $\displaystyle G_{1}( z) ,\ G_{2}( z) ,\ G_{3}( z) ,\ G_{4}( z)$:
\begin{equation*}
\Re G_{1}( z) =\sum _{ \begin{array}{l}
|t_{n} |< \frac{r}{\sqrt{2}}\\
\Re c_{n}  >0
\end{array}}\underbrace{r^{2} \Re c_{n}}_{ >0}\underbrace{\Re \left(\frac{1}{\left( r^{2} -z\overline{t_{n}}\right)^{2}}\right)}_{ >0}  >0,
\end{equation*}
\begin{equation*}
\Re G_{2}( z) =\sum _{ \begin{array}{l}
|t_{n} |< \frac{r}{\sqrt{2}}\\
\Re c_{n} < 0
\end{array}}\underbrace{r^{2} \Re c_{n}}_{< 0}\underbrace{\Re \left(\frac{1}{\left( r^{2} -z\overline{t_{n}}\right)^{2}}\right)}_{ >0} < 0,
\end{equation*}
\begin{equation*}
\Im G_{3}( z) =\sum _{ \begin{array}{l}
|t_{n} |< \frac{r}{\sqrt{2}}\\
\Im c_{n}  >0
\end{array}}\underbrace{r^{2} \Im c_{n}}_{ >0}\underbrace{\Re \left(\frac{1}{\left( r^{2} -z\overline{t_{n}}\right)^{2}}\right)}_{ >0}  >0,
\end{equation*}
\begin{equation*}
\Im G_{4}( z) =\sum _{ \begin{array}{l}
|t_{n} |< \frac{r}{\sqrt{2}}\\
\Im c_{n} < 0
\end{array}}\underbrace{r^{2} \Im c_{n}}_{< 0}\underbrace{\Re \left(\frac{1}{\left( r^{2} -z\overline{t_{n}}\right)^{2}}\right)}_{ >0} < 0.
\end{equation*}
Therefore, by Smirnov's lemma, the estimate holds.
\begin{gather*}
\int\limits _{0}^{2\pi }\left| G_{1}\left( re^{i\varphi }\right)\right| ^{p} d\varphi \leqslant \frac{2\pi }{\cos\frac{\pi p}{2}} |G_{1}( 0) |^{p} \leqslant \frac{2\pi }{\cos\frac{\pi p}{2}}\left| \sum _{ \begin{array}{l}
|t_{n} |< \frac{r}{\sqrt{2}}\\
\Re c_{n}  >0
\end{array}}\frac{\Re c_{n}}{r^{2}}\right| ^{p} \leqslant \\
\leqslant \frac{2\pi }{\cos\frac{\pi p}{2}}\left(\frac{1}{r^{2}}\sum _{ \begin{array}{l}
|t_{n} |< \frac{r}{\sqrt{2}}\\
\Re c_{n}  >0
\end{array}} |\Re c_{n} |\right)^{p} \leqslant \frac{2\pi }{\cos\frac{\pi p}{2}}\left(\frac{1}{r^{2}}\sum _{|t_{n} |< \frac{r}{\sqrt{2}}} |c_{n} |\right)^{p}.
\end{gather*}
Similarly, we obtain estimates
\begin{equation*}
\int\limits _{0}^{2\pi }\left| G_{j}\left( re^{i\varphi }\right)\right| ^{p} d\varphi \leqslant \frac{2\pi }{\cos\frac{\pi p}{2}}\left(\frac{1}{r^{2}}\sum _{|t_{n} |< \frac{r}{\sqrt{2}}} |c_{n} |\right)^{p} ,\ j=2,3,4.
\end{equation*}
Now, let's estimate the entire integral using the inequality $\displaystyle |x+y|^{p} \leqslant |x|^{p} +|y|^{p}$ as $\displaystyle 0< p< 1$:
\begin{gather*}
\int\limits _{0}^{2\pi }\left| \sum _{|t_{n} |< \frac{r}{\sqrt{2}}}\frac{c_{n}}{\left( re^{i\varphi } -t_{n}\right)^{2}}\right| ^{p} d\varphi =\int\limits _{0}^{2\pi }\left| \sum _{|t_{n} |< \frac{r}{\sqrt{2}}}\frac{\overline{c_{n}} r^{2}}{\left( r^{2} -z\overline{t_{n}}\right)^{2}}\right| ^{p} d\varphi =\int\limits _{0}^{2\pi }\left| \sum _{j=1}^{4} G_{j}\left( re^{i\varphi }\right)\right| ^{p} d\varphi \leqslant \\
\leqslant \sum _{j=1}^{4}\int\limits _{0}^{2\pi }\left| G_{j}\left( re^{i\varphi }\right)\right| ^{p} d\varphi \leqslant \frac{8\pi }{\cos\frac{\pi p}{2}}\left(\frac{1}{r^{2}}\sum\limits _{|t_{n} |< \frac{r}{\sqrt{2}}} |c_{n} |\right)^{p}\xrightarrow[r\rightarrow \infty ]{} 0,
\end{gather*}
since the estimate applies
\begin{gather*}
\frac{1}{r^{2}}\sum\limits _{|t_{n} |< \frac{r}{\sqrt{2}}} |c_{n} |=\frac{1}{r^{2}}\sum\limits _{|t_{n} |< \sqrt{r}} |c_{n} |+\frac{1}{r^{2}}\sum\limits _{\sqrt{r} < |t_{n} |< \frac{r}{\sqrt{2}}} |c_{n} |=\frac{1}{r^{2}}\sum\limits _{|t_{n} |< \sqrt{r}}\frac{|c_{n} |}{|t_{n} |^{2}} \cdot \underbrace{|t_{n} |^{2}}_{< r} +\frac{1}{r^{2}}\sum\limits _{\sqrt{r} \leqslant |t_{n} |< \frac{r}{\sqrt{2}}}\frac{|c_{n} |}{|t_{n} |^{2}} \cdot \underbrace{|t_{n} |^{2}}_{< r^{2} /2} < \\
< \frac{1}{r}\sum\limits _{|t_{n} |< \sqrt{r}}\frac{|c_{n} |}{|t_{n} |^{2}} +\frac{1}{2}\sum\limits _{\sqrt{r} \leqslant |t_{n} |< \frac{r}{\sqrt{2}}}\frac{|c_{n} |}{|t_{n} |^{2}} =O\left(\frac{1}{r}\right) +o( 1) =o( 1).  
\end{gather*}
\end{proof}
Now remains the integral from the part of the series where the poles are confined within a ring around the circle over which we are integrating. For this case, I haven't been able to apply the previous technique yet, so I've come up with a new way to obtain an implicit estimate. To do this, let's define an averaging operator on functions of our type.
\begin{definition}
Let's define a linear operator
\begin{equation*}
\mathcal{J}( f) =\frac{f\left(\sqrt{z}\right) -f\left( -\sqrt{z}\right)}{4\sqrt{z}} \ -\ \text{the averaging operator}.
\end{equation*}
Then, for the derivatives of sums of Cauchy kernels
\begin{equation*}
\mathcal{J}\left(\sum _{n=1}^{\infty }\frac{c_{n}}{( z-t_{n})^{2}}\right) =\sum _{n=1}^{\infty }\frac{c_{n} t_{n}}{\left( z-t_{n}^{2}\right)^{2}}.
\end{equation*}
\end{definition}
The averaging operator $\displaystyle \mathcal{J}$ is bounded in some sense.

\begin{lema}[The boundedness of the averaging operator]\label{lem:ogrOpUsredne}
Let $\displaystyle r >0$, $\displaystyle 0< p< 1$, $\displaystyle f:\mathbb{C}\rightarrow \mathbb{C}$ be measurable, then
\begin{equation*}
\int\limits _{0}^{2\pi }| \mathcal{J}( f)| ^{p}\left( re^{i\varphi }\right) d\varphi \leqslant \frac{2^{1-2p}}{r^{p/2}}\int\limits _{0}^{2\pi } |f|^{p}\left(\sqrt{r} e^{i\varphi }\right) d\varphi .
\end{equation*}
\end{lema}
\begin{proof}
Let's note that
\begin{gather*}
\int\limits _{0}^{2\pi }| \mathcal{J}( f)| ^{p}\left( re^{i\varphi }\right) d\varphi =\int\limits _{0}^{2\pi }\left| \frac{f\left(\sqrt{r} e^{i\varphi }\right) -f\left( -\sqrt{r} e^{i\varphi }\right)}{4\sqrt{r} e^{i\varphi }}\right| ^{p} d\varphi \leqslant \frac{1}{4^{p} r^{p/2}}\int\limits _{0}^{2\pi }\left| f\left(\sqrt{r} e^{i\varphi /2}\right)\right| ^{p} +\left| -\sqrt{r} e^{i\varphi /2}\right| ^{p} d\varphi =\\
=\frac{1}{4^{p} r^{p/2}}\int\limits _{0}^{2\pi }\left| f\left(\sqrt{r} e^{i\varphi /2}\right)\right| ^{p} d\varphi +\frac{1}{4^{p} r^{p/2}}\int\limits _{0}^{2\pi }\left| f\left( -\sqrt{r} e^{i\varphi /2}\right)\right| ^{p} d\varphi =\begin{bmatrix}
\frac{\varphi }{2} =\vartheta \\
d\varphi =2d\vartheta 
\end{bmatrix} =\\
=\frac{2^{1-2p}}{r^{p/2}}\int\limits _{0}^{\pi }\left| f\left(\sqrt{r} e^{i\vartheta }\right)\right| ^{p} d\vartheta +\frac{2^{1-2p}}{r^{p/2}}\int\limits _{0}^{\pi }\left| f\left( -\sqrt{r} e^{i\vartheta }\right)\right| ^{p} d\vartheta =\\
=\frac{2^{1-2p}}{r^{p/2}}\int\limits _{0}^{2\pi }\left| f\left(\sqrt{r} e^{i\vartheta }\right)\right| ^{p} d\vartheta.
\end{gather*}
\end{proof}
The averaging operator is a proper analogue of integrating functions of the form $\displaystyle \sum _{n=1}^{\infty }\frac{c_{n}}{( z-t_{n})^{2}}$. Ordinary integration pushes these functions outside of this class, creating a sum of Cauchy kernels.
\begin{equation*}
\int \sum _{n=1}^{\infty }\frac{c_{n}}{( z-t_{n})^{2}} dz=-\sum _{n=1}^{\infty }\frac{c_{n}}{z-t_{n}} +const.
\end{equation*}
And averaging operator $\displaystyle \mathcal{J}$ takes a weighted sum of values, like integration, but does not move functions out of this class - the derivatives of sums of Cauchy kernels. Now let's prove the lemma that estimates the integral we need in the case of one term.
\begin{lema}\label{lema:IntTrivEst}
Let $\displaystyle t\in \mathbb{C}$ and $\displaystyle r >0$, $\displaystyle 0< p< \frac{1}{2}$, then
\begin{equation*}
\int\limits _{0}^{2\pi }\frac{d\varphi }{|re^{i\varphi } -t|^{2p}} \leqslant \frac{2\pi }{1-2p} \cdot \frac{1}{r^{2p}}.
\end{equation*}
\end{lema}
\begin{proof}
Let's note that by a change of variable, we can consider $\displaystyle t$ to be real
\begin{equation*}
\int\limits _{0}^{2\pi }\frac{d\varphi }{|re^{i\varphi } -t|^{2p}} =\int\limits _{0}^{2\pi }\frac{d\varphi }{\left| re^{i\varphi } -|t|e^{i\arg t}\right| ^{2p}} =\int\limits _{0}^{2\pi }\frac{d\varphi }{\left| re^{i( \varphi -\arg t)} -|t|\right| ^{2p}} =\begin{bmatrix}
\varphi -\arg t=\vartheta \\
d\varphi =d\vartheta 
\end{bmatrix} =\int\limits _{0}^{2\pi }\frac{d\vartheta }{\left| re^{i\vartheta } -|t|\right| ^{2p}}
\end{equation*}
Now
\begin{equation*}
|re^{i\vartheta } -|t||^{2} =( r\cos \vartheta -|t|)^{2} +r^{2}\sin^{2} \vartheta \geqslant r^{2}\sin^{2} \vartheta .
\end{equation*}
Therefore, let's estimate the integral, taking into account the inequality $\displaystyle \sin \vartheta \geqslant \frac{2}{\pi } \vartheta $ на $\displaystyle \left( 0;\frac{\pi }{2}\right)$:
\begin{gather*}
\int\limits _{0}^{2\pi }\frac{d\vartheta }{\left| re^{i\vartheta } -|t|\right| ^{2p}} \leqslant \frac{1}{r^{2p}}\int\limits _{0}^{2\pi }\frac{d\vartheta }{\sin^{2p} \vartheta } =\frac{4}{r^{2p}}\int\limits _{0}^{\frac{\pi }{2}}\frac{d\vartheta }{\sin^{2p} \vartheta } \leqslant \\
\leqslant \frac{2^{2-2p} \pi ^{2p}}{r^{2p}}\int\limits _{0}^{\frac{\pi }{2}}\frac{d\vartheta }{\vartheta ^{2p}} =\frac{4}{r^{2p}}\left(\frac{\pi }{2}\right)^{2p} \cdot \frac{1}{1-2p}\left(\frac{\pi }{2}\right)^{1-2p} =\frac{2\pi }{1-2p} \cdot \frac{1}{r^{2p}}.
\end{gather*}
\end{proof}
 Now, using a clever averaging method, let's first prove a stronger estimate, albeit with higher demands
\begin{lema}[The midpoint of the series with a restriction]\label{lema:Middle1}
Let $\displaystyle \{c_{n}\}_{n\in \mathbb{N}} ,\{t_{n}\}_{n\in \mathbb{N}} \subset \mathbb{C}$ such that $\displaystyle \lim _{n\rightarrow \infty } |t_{n} |=\infty $ and $\displaystyle \{t_{n}\}_{n\in \mathbb{N}}$ has no finite limit points, and
\begin{gather*}
\sum _{n=1}^{\infty }\frac{|c_{n} |}{|t_{n} |} < +\infty \ - \text{restriction},\\
\{t_{n}\}_{n\in \mathbb{N}} \ \text{has a finite convergence exponent}.
\end{gather*}
Then for each $\displaystyle 0< p< \frac{1}{2}$ and for each $\displaystyle \varepsilon  >0$
\begin{equation*}
\int\limits _{0}^{2\pi }\left| \sum _{\frac{r}{\sqrt{2}} \leqslant |t_{n} |\leqslant r\sqrt{2}}\frac{c_{n}}{\left( re^{i\varphi } -t_{n}\right)^{2}}\right| ^{p} d\varphi =o\left(\frac{1}{r^{p-\varepsilon }}\right) \ \text{при} \ r\rightarrow \infty.
\end{equation*}
\end{lema}
\begin{proof}
Let $\displaystyle \rho < +\infty $ be the convergence rate of $\displaystyle \{t_{n}\}_{n\in \mathbb{N}}$. Then let's first obtain a simple estimate using the lemma on the estimation of the integral (\ref{lema:IntTrivEst}).
\begin{gather*}
\mathcal{X}( r) =\int\limits _{0}^{2\pi }\left| \sum _{\frac{r}{\sqrt{2}} \leqslant |t_{n} |\leqslant r\sqrt{2}}\frac{c_{n}}{\left( re^{i\varphi } -t_{n}\right)^{2}}\right| ^{p} d\varphi \leqslant \int\limits _{0}^{2\pi }\sum _{\frac{r}{\sqrt{2}} \leqslant |t_{n} |\leqslant r\sqrt{2}}\frac{|c_{n} |^{p}}{|re^{i\varphi } -t_{n} |^{2p}} d\varphi =\sum _{\frac{r}{\sqrt{2}} \leqslant |t_{n} |\leqslant r\sqrt{2}} |c_{n} |^{p}\int\limits _{0}^{2\pi }\frac{d\varphi }{|re^{i\varphi } -t_{n} |^{2p}} \leqslant \\
\leqslant \frac{2\pi }{1-2p}\frac{1}{r^{2p}}\sum _{\frac{r}{\sqrt{2}} \leqslant |t_{n} |\leqslant r\sqrt{2}} |c_{n} |^{p} =\frac{2\pi }{1-2p}\frac{1}{r^{2p}}\sum _{\frac{r}{\sqrt{2}} \leqslant |t_{n} |\leqslant r\sqrt{2}}\frac{|c_{n} |^{p}}{|t_{n} |^{p+\rho +1}} \cdot|t_{n} |^{p+\rho +1}\lesssim \\
\lesssim r^{\rho +1-p}\sum _{\frac{r}{\sqrt{2}} \leqslant |t_{n} |\leqslant r\sqrt{2}}\frac{|c_{n} |^p}{|t_{n} |^p}\frac{1}{|t_{n} |^{\rho +1}} \leqslant r^{\rho +1-p}\underbrace{\sup\limits _{n\in \mathbb{N}}\frac{|c_{n} |^p}{|t_{n} |^p}}_{< +\infty } \cdot \underbrace{\sum\limits _{\frac{r}{\sqrt{2}} \leqslant |t_{n} |\leqslant r\sqrt{2}}\frac{1}{|t_{n} |^{\rho +1}}}_{=o( 1)} =o\left( r^{\rho +1-p}\right) =o\left( r^{\rho +1}\right)
\end{gather*}
Therefore, $\displaystyle \mathcal{X}( r) =o\left( r^{\rho +1}\right)$ for \textcolor[rgb]{0.82,0.01,0.11}{\textbf{all}} sequences $\displaystyle \{c_{n}\}_{n\in \mathbb{N}} ,\{t_{n}\}_{n\in \mathbb{N}}$ satisfying the conditions, now let's improve the estimate. Since
\begin{equation*}
\mathcal{J}\left(\sum _{\frac{r}{\sqrt{2}} \leqslant |t_{n} |\leqslant r\sqrt{2}}\frac{\frac{c_{n}}{\sqrt{t_{n}}}}{\left( z-\sqrt{t_{n}}\right)^{2}}\right) =\sum _{\frac{r}{\sqrt{2}} \leqslant |t_{n} |\leqslant r\sqrt{2}}\frac{c_{n}}{( z-t_{n})^{2}},
\end{equation*}
then by the lemma on the boundedness of averaging operator (\ref{lem:ogrOpUsredne})
\begin{equation*}
\mathcal{X}( r) =\int\limits _{0}^{2\pi }\left| \sum _{\frac{r}{\sqrt{2}} \leqslant |t_{n} |\leqslant r\sqrt{2}}\frac{c_{n}}{\left( re^{i\varphi } -t_{n}\right)^{2}}\right| ^{p} d\varphi \leqslant \frac{2^{1-2p}}{r^{p/2}}\underbrace{\int\limits _{0}^{2\pi }\left| \sum\limits _{\frac{r}{\sqrt{2}} \leqslant |t_{n} |\leqslant r\sqrt{2}}\frac{\frac{c_{n}}{\sqrt{t_{n}}}}{\left(\sqrt{r} e^{i\varphi } -\sqrt{t_{n}}\right)^{2}}\right| ^{p} d\varphi }_{=o\left( r^{\frac{\rho +1}{2}}\right) \ \text{by what has been proven} } =o\left( r^{\frac{\rho +1-p}{2}}\right).
\end{equation*}
Here, I used the generality of the estimate $\displaystyle \mathcal{X}( r) =o\left( r^{\rho +1}\right)$ and obtained a new estimate $\displaystyle \mathcal{X}( r) =o\left( r^{\frac{\rho +1-p}{2}}\right)$. Now, let's denote $\displaystyle d( \alpha ) =\frac{\alpha -p}{2}$, then repeating the process $\displaystyle n$ times, we get the estimate
\begin{equation*}
\mathcal{X}( r) =o\left( r^{\overbrace{d\circ \dotsc \circ d}^{n\ \text{раз}}( \rho +1)}\right) =o\left( r^{-p+\frac{\rho +1+p}{2^{n}}}\right).
\end{equation*}
Let's choose a large $\displaystyle n$ such that $\displaystyle \frac{\rho +1+p}{2^{n}} < \varepsilon $, then ultimately
\begin{equation*}
\mathcal{X}( r) =o\left(\frac{1}{r^{p-\varepsilon }}\right) \ \text{при} \ r\rightarrow \infty.
\end{equation*}
\end{proof}
 Now let's prove the main theorem by combining all the estimates.
\begin{thm}
Let $\displaystyle \{c_{n}\}_{n\in \mathbb{N}} ,\{t_{n}\}_{n\in \mathbb{N}} \subset \mathbb{C}$ such that $\displaystyle \lim _{n\rightarrow \infty } |t_{n} |=\infty $ and $\displaystyle \{t_{n}\}_{n\in \mathbb{N}}$ has no finite limit points, and
\begin{equation*}
\{t_{n}\}_{n\in \mathbb{N}} \ \text{has a finite convergence exponent}.
\end{equation*}
Then for each $\displaystyle 0< p< \frac{1}{2}$ and for each $\displaystyle \varepsilon  >0$
\begin{equation*}
\sum _{n=1}^{\infty }\frac{|c_{n} |}{|t_{n} |^{2}} < +\infty \ \Longrightarrow \int\limits _{0}^{2\pi }\left| \sum _{n=1}^{\infty }\frac{c_{n}}{\left( re^{i\varphi } -t_{n}\right)^{2}}\right| ^{p} d\varphi =o\left( r^{\varepsilon }\right) \ \text{при}\ r\rightarrow \infty,
\end{equation*}
\begin{equation*}
\sum _{n=1}^{\infty }\frac{|c_{n} |}{|t_{n} |} < +\infty \ \Longrightarrow \int\limits _{0}^{2\pi }\left| \sum _{n=1}^{\infty }\frac{c_{n}}{\left( re^{i\varphi } -t_{n}\right)^{2}}\right| ^{p} d\varphi =o\left(\frac{1}{r^{p-\varepsilon }}\right) \ \text{при} \ r\rightarrow \infty, 
\end{equation*}
\begin{equation*}
\sum _{n=1}^{\infty } |c_{n} |< +\infty \ \Longrightarrow \int\limits _{0}^{2\pi }\left| \sum _{n=1}^{\infty }\frac{c_{n}}{\left( re^{i\varphi } -t_{n}\right)^{2}}\right| ^{p} d\varphi =o\left(\frac{1}{r^{2p-\varepsilon }}\right) \ \text{при} \ r\rightarrow \infty. 
\end{equation*}
\end{thm}
\begin{proof}
\begin{enumerate}
\item Let $\displaystyle \sum _{n=1}^{\infty }\frac{|c_{n} |}{|t_{n} |} < +\infty $. Then, using the lemmas (\ref{lem:tail}), (\ref{lem:start}), (\ref{lema:Middle1}) about the estimates of integrals
\begin{gather*}
\int\limits _{0}^{2\pi }\left| \sum _{n=1}^{\infty }\frac{c_{n}}{\left( re^{i\varphi } -t_{n}\right)^{2}}\right| ^{p} d\varphi \leqslant \int\limits _{0}^{2\pi }\left| \sum _{|t_{n} |< \frac{r}{\sqrt{2}}}\frac{c_{n}}{\left( re^{i\varphi } -t_{n}\right)^{2}}\right| ^{p} d\varphi +\int\limits _{0}^{2\pi }\left| \sum _{\frac{r}{\sqrt{2}} \leqslant |t_{n} |\leqslant r\sqrt{2}}\frac{c_{n}}{\left( re^{i\varphi } -t_{n}\right)^{2}}\right| ^{p} d\varphi +\\
+\int\limits _{0}^{2\pi }\left| \sum _{|t_{n} | >r\sqrt{2}}\frac{c_{n}}{\left( re^{i\varphi } -t_{n}\right)^{2}}\right| ^{p} d\varphi \lesssim \left(\frac{1}{r^{2}}\sum\limits _{|t_{n} |< \frac{r}{\sqrt{2}}} |c_{n} |\right)^{p} +o\left(\frac{1}{r^{p-\varepsilon }}\right) +\left(\sum\limits _{|t_{n} | >r\sqrt{2}}\frac{|c_{n} |}{|t_{n} |^{2}}\right)^{p}.
\end{gather*}
Let's estimate these sums
\begin{gather*}
\frac{1}{r^{2}}\sum\limits _{|t_{n} |< \frac{r}{\sqrt{2}}} |c_{n} |=\frac{1}{r^{2}}\sum\limits _{|t_{n} |< \sqrt{r}} |c_{n} |+\frac{1}{r^{2}}\sum\limits _{\sqrt{r} < |t_{n} |< \frac{r}{\sqrt{2}}} |c_{n} |=\frac{1}{r^{2}}\sum\limits _{|t_{n} |< \sqrt{r}}\frac{|c_{n} |}{|t_{n} |} \cdot \underbrace{|t_{n} |}_{< \sqrt{r}} +\frac{1}{r^{2}}\sum\limits _{\sqrt{r} \leqslant |t_{n} |< \frac{r}{\sqrt{2}}}\frac{|c_{n} |}{|t_{n} |} \cdot \underbrace{|t_{n} |}_{< r/\sqrt{2}} \leqslant \\
\leqslant \frac{1}{r\sqrt{r}}\sum\limits _{|t_{n} |< \sqrt{r}}\frac{|c_{n} |}{|t_{n} |} +\frac{1}{r\sqrt{2}}\sum\limits _{\sqrt{r} \leqslant |t_{n} |< \frac{r}{\sqrt{2}}}\frac{|c_{n} |}{|t_{n} |} =\frac{1}{r}\left(\frac{1}{\sqrt{r}}\underbrace{\sum\limits _{|t_{n} |< \sqrt{r}}\frac{|c_{n} |}{|t_{n} |}}_{=O( 1)} +\frac{1}{\sqrt{2}}\underbrace{\sum\limits _{\sqrt{r} \leqslant |t_{n} |< \frac{r}{\sqrt{2}}}\frac{|c_{n} |}{|t_{n} |}}_{=o( 1)}\right) =o\left(\frac{1}{r}\right)
\end{gather*}
and
\begin{equation*}
\sum\limits _{|t_{n} | >r\sqrt{2}}\frac{|c_{n} |}{|t_{n} |^{2}} =\sum\limits _{|t_{n} | >r\sqrt{2}}\frac{|c_{n} |}{|t_{n} |} \cdotp \frac{1}{\underbrace{|t_{n} |}_{ >r\sqrt{2}}} < \frac{1}{r\sqrt{2}}\underbrace{\sum\limits _{|t_{n} | >r\sqrt{2}}\frac{|c_{n} |}{|t_{n} |}}_{=o( 1)} =o\left(\frac{1}{r}\right).
\end{equation*}
Therefore, we obtain an estimate for the original integral.
\begin{equation*}
\int\limits _{0}^{2\pi }\left| \sum _{n=1}^{\infty }\frac{c_{n}}{\left( re^{i\varphi } -t_{n}\right)^{2}}\right| ^{p} d\varphi =o\left(\frac{1}{r^{p}}\right) +o\left(\frac{1}{r^{p-\varepsilon }}\right) +o\left(\frac{1}{r^{p}}\right) =o\left(\frac{1}{r^{p-\varepsilon }}\right).
\end{equation*}
\item Now, let $\displaystyle \sum _{n=1}^{\infty }\frac{|c_{n} |}{|t_{n} |^{2}} < +\infty $. Then, similarly using the lemmas (\ref{lem:tail}), (\ref{lem:start}) about the estimates of integrals
\begin{equation*}
\int\limits _{0}^{2\pi }\left| \sum _{n=1}^{\infty }\frac{c_{n}}{\left( re^{i\varphi } -t_{n}\right)^{2}}\right| ^{p} d\varphi \lesssim \left(\frac{1}{r^{2}}\sum\limits _{|t_{n} |< \frac{r}{\sqrt{2}}} |c_{n} |\right)^{p} +\left(\sum\limits _{|t_{n} | >r\sqrt{2}}\frac{|c_{n} |}{|t_{n} |^{2}}\right)^{p} +\int\limits _{0}^{2\pi }\left| \sum _{\frac{r}{\sqrt{2}} \leqslant |t_{n} |\leqslant r\sqrt{2}}\frac{c_{n}}{\left( re^{i\varphi } -t_{n}\right)^{2}}\right| ^{p} d\varphi.
\end{equation*}
Let's estimate the sums in a similar manner
\begin{gather*}
\frac{1}{r^{2}}\sum\limits _{|t_{n} |< \frac{r}{\sqrt{2}}} |c_{n} |=\frac{1}{r^{2}}\sum\limits _{|t_{n} |< \sqrt{r}} |c_{n} |+\frac{1}{r^{2}}\sum\limits _{\sqrt{r} < |t_{n} |< \frac{r}{\sqrt{2}}} |c_{n} |=\frac{1}{r^{2}}\sum\limits _{|t_{n} |< \sqrt{r}}\frac{|c_{n} |}{|t_{n} |^{2}} \cdot \underbrace{|t_{n} |^{2}}_{< r} +\frac{1}{r^{2}}\sum\limits _{\sqrt{r} \leqslant |t_{n} |< \frac{r}{\sqrt{2}}}\frac{|c_{n} |}{|t_{n} |^{2}} \cdot \underbrace{|t_{n} |^{2}}_{< r^{2} /2} \leqslant \\
\leqslant \frac{1}{r}\underbrace{\sum\limits _{|t_{n} |< \sqrt{r}}\frac{|c_{n} |}{|t_{n} |^{2}}}_{=O( 1)} +\frac{1}{2}\underbrace{\sum\limits _{\sqrt{r} \leqslant |t_{n} |< \frac{r}{\sqrt{2}}}\frac{|c_{n} |}{|t_{n} |^{2}}}_{=o( 1)} =o( 1)
\end{gather*}
and
\begin{equation*}
\sum\limits _{|t_{n} | >r\sqrt{2}}\frac{|c_{n} |}{|t_{n} |^{2}} =o(1).
\end{equation*}
However, the remaining integral cannot be estimated using lemma (\ref{lema:Middle1}) because we have a weaker condition $\displaystyle \sum _{n=1}^{\infty }\frac{|c_{n} |}{|t_{n} |^{2}} < +\infty$ instead of $\displaystyle \sum _{n=1}^{\infty }\frac{|c_{n} |}{|t_{n} |} < +\infty$. However, it is better to represent the function as
\begin{equation*}
\sum \frac{c_{n}}{( z-t_{n})^{2}} =\sum \frac{c_{n}}{t_{n}^{2}} -z\cdot \sum \frac{\frac{c_{n}}{t_{n}^{2}}}{z-t_{n}} +z\cdot \sum \frac{\frac{c_{n}}{t_{n}}}{( z-t_{n})^{2}}
\end{equation*}
and now let's estimate the integral \footnote{In Ostrovsky's theorem (\ref{eq:Ostrov}) and in the lemmas above, it was assumed that $\displaystyle c_n$ does not depend on $\displaystyle r$. However, the estimates in the theorems will still hold true if $\displaystyle c_n=c_n(r)$, and the statements about asymptotics will also remain true if, for example, we replace the convergence condition $\displaystyle \sum<+\infty$ with $\sup\limits_{r}\sum<+\infty$.}
\begin{gather*}
\int\limits _{0}^{2\pi }\left| \sum _{\frac{r}{\sqrt{2}} \leqslant |t_{n} |\leqslant r\sqrt{2}}\frac{c_{n}}{\left( re^{i\varphi } -t_{n}\right)^{2}}\right| ^{p} d\varphi \leqslant 2\pi \underbrace{\left|\sum\limits _{\frac{r}{\sqrt{2}} \leqslant |t_{n} |\leqslant r\sqrt{2}}\frac{c_{n}}{t_{n}^{2}}\right|^p }_{=o( 1)} +r^{p}\underbrace{\int\limits _{0}^{2\pi }\left| \sum\limits _{n=1}^{\infty }\frac{\frac{c_{n}}{t_{n}^{2}} \cdot \mathbb{1}_{\left\{r/\sqrt{2} \leqslant |t_{n} |\leqslant r\sqrt{2}\right\}}}{re^{i\varphi } -t_{n}}\right| ^{p} d\varphi }_{=o( 1) \ \ \text{by Ostrovsky theorem (\ref{eq:Ostrov})}} +\\
+r^{p}\underbrace{\int\limits _{0}^{2\pi }\left| \sum\limits _{n=1}^{\infty }\frac{\frac{c_{n}}{t_{n}} \cdot \mathbb{1}_{\left\{r/\sqrt{2} \leqslant |t_{n} |\leqslant r\sqrt{2}\right\}}}{\left( re^{i\varphi } -t_{n}\right)^{2}}\right| }_{=o\left( 1/r^{p-\varepsilon }\right) \ \text{by lemma (\ref{lema:Middle1})}}^{p} d\varphi =o\left( r^{\varepsilon }\right).
\end{gather*}
Therefore, we obtain an estimate for the original integral
\begin{equation*}
\int\limits _{0}^{2\pi }\left| \sum _{n=1}^{\infty }\frac{c_{n}}{\left( re^{i\varphi } -t_{n}\right)^{2}}\right| ^{p} d\varphi =o( 1) +o( 1) +o\left( r^{\varepsilon }\right) =o\left( r^{\varepsilon }\right).
\end{equation*}
\item Now, let $\displaystyle \sum _{n=1}^{\infty } |c_{n} |< +\infty $. Then, similarly using the lemmas (\ref{lem:tail}), (\ref{lem:start}) about the estimates of integrals
\begin{equation*}
\int\limits _{0}^{2\pi }\left| \sum _{n=1}^{\infty }\frac{c_{n}}{\left( re^{i\varphi } -t_{n}\right)^{2}}\right| ^{p} d\varphi \lesssim \left(\frac{1}{r^{2}}\sum\limits _{|t_{n} |< \frac{r}{\sqrt{2}}} |c_{n} |\right)^{p} +\left(\sum\limits _{|t_{n} | >r\sqrt{2}}\frac{|c_{n} |}{|t_{n} |^{2}}\right)^{p} +\int\limits _{0}^{2\pi }\left| \sum _{\frac{r}{\sqrt{2}} \leqslant |t_{n} |\leqslant r\sqrt{2}}\frac{c_{n}}{\left( re^{i\varphi } -t_{n}\right)^{2}}\right| ^{p} d\varphi.
\end{equation*}
Let's estimate the sums as
\begin{equation*}
\frac{1}{r^{2}}\sum\limits _{|t_{n} |< \frac{r}{\sqrt{2}}} |c_{n} |=O\left(\frac{1}{r^{2}}\right)
\end{equation*}
and
\begin{equation*}
\sum\limits _{|t_{n} | >r\sqrt{2}}\frac{|c_{n} |}{|t_{n} |^{2}} =\sum\limits _{|t_{n} | >r\sqrt{2}} |c_{n} |\cdot \frac{1}{\underbrace{|t_{n} |^{2}}_{ >2r^{2}}} \leqslant \frac{1}{2r^{2}}\sum\limits _{|t_{n} | >r\sqrt{2}} |c_{n} |=o\left(\frac{1}{r^{2}}\right).
\end{equation*}
To estimate the remaining integral, one can, of course, apply the theorem above, but the estimate will be weak. Therefore we note that
\begin{equation*}
\sum \frac{c_{n}}{( z-t_{n})^{2}} =\frac{1}{z^{2}}\sum c_{n} +\frac{2}{z^{2}} \cdot \sum \frac{c_{n} t_{n}}{z-t_{n}} +\frac{1}{z^{2}} \cdot \sum \frac{c_{n} t_{n}^{2}}{( z-t_{n})^{2}},
\end{equation*}
whence
\begin{gather*}
\int\limits _{0}^{2\pi }\left| \sum _{\frac{r}{\sqrt{2}} \leqslant |t_{n} |\leqslant r\sqrt{2}}\frac{c_{n}}{\left( re^{i\varphi } -t_{n}\right)^{2}}\right| ^{p} d\varphi \leqslant \frac{2\pi }{r^{2p}} \underbrace{\left|\sum\limits _{\frac{r}{\sqrt{2}} \leqslant |t_{n} |\leqslant r\sqrt{2}} c_{n}\right|^p}_{=o( 1)} +\frac{2}{r^{2p}}\underbrace{\int\limits _{0}^{2\pi }\left| \sum\limits _{n=1}^{\infty }\frac{c_{n} t_{n} \cdot \mathbb{1}_{\left\{r/\sqrt{2} \leqslant |t_{n} |\leqslant r\sqrt{2}\right\}}}{re^{i\varphi } -t_{n}}\right| ^{p} d\varphi }_{=o( 1) \ \text{теорема Островского (\ref{eq:Ostrov})}} +\\
+\frac{1}{r^{2p}}\underbrace{\int\limits _{0}^{2\pi }\left| \sum\limits _{n=1}^{\infty }\frac{c_{n} t_{n}^{2} \cdot \mathbb{1}_{\left\{r/\sqrt{2} \leqslant |t_{n} |\leqslant r\sqrt{2}\right\}}}{\left( re^{i\varphi } -t_{n}\right)^{2}}\right| ^{p} d\varphi }_{=o\left( r^{\varepsilon }\right) \ \text{лемма (\ref{lema:Middle1})}} =o\left(\frac{1}{r^{2p-\varepsilon }}\right).
\end{gather*}
Therefore, we obtain an estimate for the original integral
\begin{equation*}
\int\limits _{0}^{2\pi }\left| \sum _{n=1}^{\infty }\frac{c_{n}}{\left( re^{i\varphi } -t_{n}\right)^{2}}\right| ^{p} d\varphi =O\left(\frac{1}{r^{2p}}\right) +o\left(\frac{1}{r^{2p}}\right) +o\left(\frac{1}{r^{2p-\varepsilon }}\right) =o\left(\frac{1}{r^{2p-\varepsilon }}\right).
\end{equation*}
\end{enumerate}
\end{proof}

\end{document}